\documentclass[11pt]{article}

\usepackage{amsmath}
\usepackage{amssymb}
\usepackage{amsthm}
\usepackage{geometry}
\usepackage{makeidx}
\usepackage{enumerate}
\usepackage{stmaryrd}
\usepackage[english]{babel}
\usepackage[T1]{fontenc}
\usepackage{lmodern}
\usepackage[utf8]{inputenc}
\usepackage{pgf,tikz}
\usetikzlibrary{arrows}
\usepackage{relsize}

\newtheorem{lemma}{Lemma}[section]
\newtheorem{corollary}[lemma]{Corollary}
\newtheorem{proposition}[lemma]{Proposition}
\newtheorem{theorem}[lemma]{Theorem}

\newtheorem{remark}[lemma]{Remark}
\numberwithin{equation}{section}

\allowdisplaybreaks

\newcommand{\R}{\mathbb{R}}

\newcommand\N{\mathbb{N}}
\newcommand\Z{\mathbb{Z}}
\def\C{\mathbb C}
\def\CC{\mathcal{C}}

\def\11{1\!\!1}
\def\f12{\frac{1}{2}}

\title{Bernstein inequalities via the heat semigroup}
\author{Rafik Imekraz, El Maati Ouhabaz 
\thanks{\noindent Univ. Bordeaux, Institut de Math\'ematiques (IMB). CNRS UMR 5251. 351,  
Cours de la Lib\'eration 33405 Talence, France.
 Rafik.Imekraz@math.u-bordeaux.fr, Elmaati.Ouhabaz@math.u-bordeaux.fr}}

\begin{document}

\date{}
\maketitle

\noindent \begin{abstract} We extend the classical Bernstein inequality  to a general setting including the Laplace-Beltrami operator,   Schr\"odinger operators and divergence form elliptic operators on Riemannian manifolds or domains. We  prove $L^p$ Bernstein inqualities as well as  a ``reverse inequality'' which  is new even for compact manifolds (with or without boundary). Such a reverse inequality can be seen as the dual of the Bernstein inequality. The heat kernel will be the backbone of our approach but we also develop new techniques. For example,  once reformulating Bernstein inequalities in a semi-classical fashion we prove and use  weak factorization of smooth functions {\it \`a la} Dixmier-Malliavin and $BMO-L^\infty$  multiplier results (in contrast to the usual $L^\infty-BMO$ ones).  Also, our approach reveals a link between the $L^p$-Bernstein inequality and the boundedness on $L^p$ of the Riesz transform. 
\end{abstract}

\tableofcontents
%\newpage

\section{Introduction and main results}

Given any trigonometric polynomial $P(x) = \sum\limits_{k=-N}^N \alpha_k e^{ikx}$, for $x\in \R$ and  $(\alpha_{-N},\dots,\alpha_N) \in \C^{2N+1}$, the classical Bernstein inequality is given as follows 
 \begin{equation}\label{bernstein0}
 \| P' \|_\infty \le N \| P \|_\infty.
 \end{equation}
 This inequality plays an important role in many areas of mathematics such as approximation theory, random trigonometric series %(cf. Kahane \cite{kahane}) 
 and random Dirichlet series. We refer to the recent survey by Queff\'elec and Zarouf \cite{queff-zarouf} for references,  history and developments of the Bernstein inequality.  

%In the present paper we extend this inequality in several directions and we also prove a dual version of that inequality, named below reverse Bernstein inequality. Our reverse Bernstein inequality (see $(RB_q)$ in Theorem \ref{theo3-p} and the end of the paper) is new even for the important case of compact Riemannian manifolds.

In the present paper we extend this inequality in several directions. The point of view is to replace the trigonometric polynomials by any combination of
eigenfunctions of a given operator $L$  on a Riemannian manifold. Here $L$ is either the Laplace-Beltrami operator on a compact manifold or a Schr\"odinger operator, or a divergence form elliptic operator on a domain. We give a necessary and sufficient condition involving the heat semigroup (see Theorem \ref{theo2-p}) so that the Bernstein inequality can be extended to this general setting.
We also prove a dual version of that inequality, named below \textit{a reverse Bernstein inequality}.  This latter inequality is new even for compact Riemannian manifolds (see $(RB_q)$ in Theorem \ref{theo3-p}). 

Before stating some of our main results and also previously known ones we start by describing the setting of the paper.

Let $M$ be a smooth connected Riemannian manifold. We  denote by $d$ and $\mu$ the Riemannian distance and measure of $M$. 
For simplicity, we use the usual notation  $V(x,r) := \mu(B(x,r))$ for any $x\in M$ and $r > 0$, where $B(x,r)$ denotes the open ball of centre $x$ and radius $r$. 
We assume that $(M,d,\mu)$ is a space of homogenous type. That is the following doubling property holds for some constant $C > 0$ 
\begin{equation*}
 V(x,2r)\leq C V(x,r) \quad  \mbox{ for all } x\in M \quad \mbox{and}\quad \forall r>0.
\end{equation*}
It follows easily from this property that there exists a positive constant $n \le \frac{\ln(C)}{\ln(2)}$ such that :
%Setting $n : =\frac{\ln(C)}{\ln(2)}$, it is easy to see that the doubling property implies the stronger one
\begin{equation}\label{doublin}
V(x,sr)\leq Cs^n V(x,r) \quad \mbox{ for all } s\geq 1.
\end{equation}
Given  a non-negative potential $W \in L^1_{loc}(M)$, 
%Let us consider a non-negative self-adjoint Schr\"{o}dinger operator $L:D(L)\rightarrow L^2(X)$ with domain $D(L)\subset L^2(X)$ of the form
we consider the Schr\"odinger operator 
\begin{equation*}
 L := -\Delta + W.
\end{equation*}
As usual, $L$ is defined through its quadratic form (see the monographs of Davies \cite{Davies89} or Ouhabaz \cite[Chapter 1]{maati}). It is a non-negative self-adjoint operator on $L^2(M)$ with an appropriate dense domain $D(L)$ containing $\CC_c^\infty(M)$.
It is a basic fact that $-L$ is the generator of a positive semigroup $(e^{-tL})_{t\geq 0}$ of contractions on $L^2(M)$. We assume that $e^{-tL}$ is given by an integral kernel $p_t(x,y)$ in the usual sense
\begin{equation*}
(e^{-tL} f )(x)=\int_{M} p_t(x,y) f(y) d\mu(y)
\end{equation*}
for every $f \in L^2(M)$ and for almost every $x \in M$. The function $p_t(x,y)$ is called the heat kernel of $L$. We assume that $p_t(x,y)$ satisfies the following Gaussian upper bound
\begin{equation*}
%\begin{array}{cccrcl}
 |p_t(x,y)|  \leq  \frac{C}{V(x,\sqrt{t})}\exp\left( -c\frac{d(x,y)^2}{t}\right) \eqno(G)
%\end{array}
\end{equation*}
for some positive constant $C$, for any $t>0$ and for almost every $(x,y)  \in M \times M$.  We recall that due to the non-negativity of the potential $W$, the Gaussian upper bound $(G)$ holds for the heat kernel of $L$ if it holds for  the heat kernel of $\Delta$ (see \cite[page 195]{maati}). The validity of $(G)$ for the Laplacian is a classical subject which has been studied for many  years, see e.g. Li and Yau \cite{liyau}, Davies \cite{Davies89}, Grigor'yan \cite{Grigo-livre} and the references therein. 

For the sake of clarity, we begin by stating our results in the case where the spectrum of the Schr\"{o}dinger operator $L$ is discrete. In this case, we denote by  $(\lambda_k^2)$ the sequence of eigenvalues of $L$ and $\phi_k$ the corresponding normalized eigenfunctions as follows
\begin{equation}\label{dis}
L\phi_k=\lambda_k^2 \phi_k,  \qquad   \left\Vert \phi_k\right\Vert_{2}=1, \qquad  0\leq \lambda_0 < \lambda_1 \leq \dots \rightarrow +\infty.
\end{equation}
The discreteness of the spectrum of $L$ holds for instance if $M$ is compact or if  $M=\R^d$ and  $W(x) \to +\infty$ as $|x |  \to + \infty$.
In addition, since $M$ is connected, it is a basic fact  that the first eigenvalue $\lambda^2_0$ is simple. Hence, $\lambda_k > 0$ for all $k \ge 1$. 

%We refer to Section \ref{semi} for the general case. Since we first restrict the presentation of our results under the assumption \eqref{dis}, one may %assume that there is $N_0\in \N$ such that $\lambda_{N_0}\geq 1$ in \eqref{dis}.
\begin{theorem}\label{theo1-p}
Suppose that the doubling condition \eqref{doublin} and the Gaussian upper bound $(G)$ are satisfied. Assume also that $L = -\Delta + W$ has discrete spectrum (with the above notations \eqref{dis}) and 
let  $p\in [1,2]$.  Then the following \textit{Bernstein inequality} holds:  there exists a constant $C >0$ such that for every $N \geq 1$ and every $(N+1)$-tuple of coefficients 
 $(\alpha_0,\dots,\alpha_N)\in \C^{N+1}$, 
\begin{equation*}
\| \nabla \Big{(}  \sum_{k=0}^N  \alpha_k \phi_k \Big{)}  \|_p  + \| \sqrt{W} \Big{(}  \sum_{k=0}^N  \alpha_k \phi_k \Big{)}  \|_p \le C \lambda_N \
\|  \Big{(}  \sum_{k=0}^N  \alpha_k \phi_k \Big{)}  \|_p. \eqno(B_p)
\end{equation*}
\end{theorem}
%Actually, the proof allows $C$ to be independent of $p$ (by interpolation).
%An interesting aspect of the previous result is that we obtain a gradient estimate without any gradient estimate assumption on the heat kernel $p_t(x,y)$. One %may compare that situation to \cite[Theorem 2.1]{filbir2010} where gradient assumptions are necessary to reach the case $p<2$. 

Our next result is for the case $p \in [2, \infty]$.  In this case, one needs a regularity property of the heat semigroup, which turns out to be also necessary.  

\begin{theorem} \label{theo2-p} Assume the doubling condition \eqref{doublin} and the Gaussian upper bound $(G)$. Assume also that $L = -\Delta + W$ has discrete spectrum (with the above notations \eqref{dis}).
For any $p \in [2, \infty]$, the Bernstein inequality $(B_p)$ is equivalent to the following regularity property 
\begin{equation*}
\| \nabla e^{-tL} \|_{p \to p} +\| \sqrt{W} e^{-tL} \|_{p \to p}  \le  \frac{C}{\sqrt{t}} \quad \mbox{ for all }t>0. \eqno(R_p) 
\end{equation*}
%In particular, if the Riesz transform type operators $\nabla L^{-1/2}$ and $\sqrt{W} L^{-1/2}$ are bounded on $L^p(M)$ then $(B_p)$ holds true.
\end{theorem}

We view our estimate $(B_p)$ as a generalization of the Bernstein inequality \eqref{bernstein0} (in $L^p$)  in which the derivative $\frac{d}{dx}$ is replaced by $\nabla$ and the trigonometric  polynomial $P(x)$ by a combination of eigenfunctions since it encompasses 
\begin{equation} \label{bern2-p} 
\| \nabla \Big{(}  \sum_{k=0}^N  \alpha_k \varphi_k \Big{)}  \|_p  \le C \lambda_N \, 
\|  \Big{(} \sum_{k=0}^N  \alpha_k \varphi_k \Big{)}  \|_p. 
\end{equation}
 
\begin{remark}
\begin{enumerate}[i)]
\item The above results as well as the forthcoming ones are also valid if one replaces $M$ by any open subset $\Omega$ of $M$. For general $\Omega$, the space $(\Omega, d, \mu)$ may not be a space of homogeneous type. The assumptions here are $(G)$ holds on a.e. $x, y \in \Omega$ and $M$ satisfies  \eqref{doublin}.
\item Theorem \ref{theo2-p} shows the link between the Bernstein inequality $(B_p)$ and the Riesz transform through the regularity property $(R_p)$. We shall make this more precise after the statement of Theorem \ref{theo2-psc}.
\item If $(R_p)$ is replaced with the weaker bound $ \| \nabla e^{-tL} \|_{p \to p} \leq\frac{C}{\sqrt{t}}$, our proof shows $(B_p)$ with the gradient term only, i.e.  \eqref{bern2-p} holds. The same remark applies to the term  $\| \sqrt{W} e^{-tL} \|_{p \to p}$. 
\end{enumerate}
\end{remark}
%%%%

The question of extending the original Bernstein inequality  by replacing 
the trigonometric polynomials by eigenfunctions of a given self-adjoint operator has been considered in the past. 
For instance, an important local version of the Bernstein inequality was obtained by Donnelly and Fefferman \cite{donnelly1990growth} on boundaryless compact Riemannian manifolds.
Motivated by a stronger version of the local Donnelly-Fefferman inequalities, Ortega-Cerd\`a and Pridhnani proved 
in \cite[Corollary 3.3]{ortega2013carleson} a Bernstein inequality for a single eigenfunction $\phi_k$ of $L=-\Delta$, namely,
\begin{equation*}
\left\Vert \nabla \phi_k\right\Vert_{\infty}\leq C \lambda_k\left\Vert \phi_k\right\Vert_{\infty}.
\end{equation*}
Actually, the general form \eqref{bern2-p} is conjectured in \cite[page 157]{ortega2013carleson} but appears to be a consequence of  \cite[Theorem 2.1]{filbir2010} by Filbir and Mhaskar and heat kernel estimates of the Laplace-Beltrami operator. The paper \cite{filbir2010} deals with operators in an abstract setting, however additional assumptions are made on the heat kernel such as  gradient estimates which are rather restrictive. 
A related result is given by Imekraz \cite[Theorem 5.3]{imek-aif} for the case of Schr\"{o}dinger operators of the form $-\Delta+|x|^{2\alpha}$ for $\alpha\in \N^\star$ on $\R^d$. Although the method of the latter paper is quite flexible, it  deals with a specific class of pseudo-differential symbols.  In contrast, the approach of the present paper is  more general and gives a unified framework  relying merely on estimates on the heat semigroup. Our results extend the results mentioned above in many directions. Not only we deal with  both terms $\nabla \Big{(}  \sum\limits_{k=0}^N  \alpha_k \phi_k \Big{)}$ and $\sqrt{W} \Big{(}  \sum\limits_{k=0}^N  \alpha_k \phi_k \Big{)}$ for Schr\"odinger operators in our estimates,  we also assume much less regularity properties on the associated heat kernel. Our method is very flexible and applies to operators on domains of $M$. It also applies to elliptic operators in divergence form 
$$L = - \sum_{k,l=1}^n \frac{\partial}{\partial x_k} \left( c_{kl}(x) \frac{\partial}{\partial x_l} \right)$$
with bounded measurable coefficients $c_{kl}$ on domains of $\R^n$ (see Section \ref{sec-examples}). 

By similar ideas we also prove 
$L^p(M)-L^q(M)$ Bernstein inequalities. More precisely,

\begin{theorem}\label{theo-LpLq}
Suppose the doubling property \eqref{doublin} and the Gaussian upper bound $(G)$.
Suppose in addition that 
\begin{equation}\label{1infty}
|p_t(x,x)|  \leq  \frac{C}{t^{m/2}} \quad \mbox{for all  }t\in (0,1] \quad \mbox{and  a.e. } x\in M
\end{equation}
for some constant $m > 0$. 
Then, for any couple $(p,q)$ satisfying $1\leq p \leq q \leq 2$, there is a constant $C>0$ such for every $N \geq 1$ and every $(N+1)$-tuple of coefficients $(\alpha_0,\dots,\alpha_N)\in \C^{N+1}$ the following inequality holds 

\begin{equation*}
\Big\| \nabla \sum_{k=0}^N \alpha_k \phi_k     \Big\|_{q} + \Big\| \sqrt{W} \sum_{k=0}^N \alpha_k \phi_k     \Big\|_{q} \le C \lambda_N^{1+ m \left\vert\frac{1}{p}-\frac{1}{q} \right\vert}  \Big\| \sum_{k=0}^N \alpha_k \phi_k \Big\|_{p}. \eqno{(B_{p,q})}
\end{equation*}
If the regularity property  $(R_{q_0})$ holds for some $q_0 \in (2, \infty]$, then  $(B_{p,q})$ holds for $1\leq p\leq q \le q_0$.
 \end{theorem}
 
Note that obviously \eqref{1infty} follows from $(G)$ if $V(x,r) \ge C r^m$ for all $x \in M$ and $r  \in (0, 1]$. On the other hand, it is known that  \eqref{1infty} for some $m \in (2,+\infty)$ and all $t > 0$ is equivalent to a global Sobolev inequality (see Varopoulos \cite{varo85} or Davies \cite{Davies89}, Section 2.4).

Now we describe briefly the strategy of the proofs.  We borrow some ideas from two different subjects which are semi-classical analysis and the theory of spectral multipliers.  Firstly, we reformulate $(B_p)$ as the  {\it semi-classical Bernstein estimate}
\begin{equation*}
\left\Vert \nabla \psi(hL) f \right\Vert_{p} + \left\Vert \sqrt{W} \psi(hL) f \right\Vert_{p} \leq  \frac{C}{\sqrt{h}} \left\Vert f\right\Vert_{p},\qquad h>0,\quad f\in L^p(M) \eqno(SB_p)
 \end{equation*}
for every given  $\psi \in \CC_c^\infty([0, \infty))$ and $C$ is a positive constant  which may depend on $\psi$. It turns out that  $(B_p)$ and $(SB_p)$ are equivalent (see Section \ref{semi}). The reformulation  $(SB_p)$ allows us to apply techniques of spectral multipliers.   Another  advantage  of $(SB_p)$ is that it can be considered without assuming the spectrum of $L$ to be discrete. We deals with $(B_{p,q})$ in Theorem \ref{theo-LpLq} in a similar way. Now, in order to prove $(SB_p)$ we write
\begin{equation*}
\psi(hL) = \int_{\R} \widehat{\psi}(\xi+2i) e^{-h(2- i\xi)L} \frac{d\xi}{\sqrt{2\pi}} 
\end{equation*}
so that we need $L^p$-estimates for $\nabla e^{-h(2- i\xi)L}$ and $\sqrt{W} e^{-h(2+ i\xi)L}$. To estimate $\nabla e^{-hL}$ and $\sqrt{W} e^{-hL}$ we use weighted $L^2$-estimates due to Grigor'yan \cite{grigo95} (proved there in the case where  $W = 0$). The $L^p$-estimates of the remaining term 
 $e^{-h(1+ i\xi)L}$ are based on the fact that the estimate $(G)$ for $t > 0$ extends to $t \in \C^+$, this fact is taken from  Carron, Coulhon and Ouhabaz \cite{maati2} (a prior related result was proved by Davies \cite{Davies89} under the assumption \eqref{1infty}). These techniques have been  used  to prove spectral multiplier results for sel-adjoint operators in Duong, Ouhabaz and Sikora \cite{DOS} (see also Ouhabaz \cite[Chapter 7]{maati}).  

The first step of the proof of Theorem \ref{theo2-p} is also based on  the semi-classical point of view. In order to prove that the regularity property $(R_p)$ is necessary we need a new argument.  It is a based on a result stating that every smooth function $f:\R\rightarrow \R$ belonging to $L^1(\R)$, as well all its derivatives, can be written as
\begin{equation}\label{dixm}
 f = \rho_1 \ast f_1 + \rho_2 \ast  f_2,
 \end{equation}
where $\rho_i \in \CC_c^\infty(\R)$ and $f_i \in L^1(\R) $ for $ i = 1, 2$. This  weak factorization is based on a lemma of Dixmier and Malliavin \cite{dixmi} and it is interesting in its own (see a precise statement in Lemma \ref{dixm2}). 

\medskip
Another contribution of the present paper is to study  lower bounds in the Bernstein inequality. Shi and Xu \cite{shi2010} prove the following upper and lower bounds  on a boundaryless compact Riemannian manifold
\begin{equation} \label{sx10}
C^{-1} \lambda_k \left\Vert \phi_k \right\Vert_{\infty} \leq   \left\Vert \nabla \phi_k\right\Vert_{\infty}\leq C  \lambda_k \left\Vert \phi_k \right\Vert_{\infty}
\end{equation}
for all $k$ and $\lambda_k \ge 1$. %The lower bound is seen  as a consequence of a prior result of Zelditch \cite{zelditch2008} on nodal sets but 
Note that \eqref{sx10} is for each single eigenfunction and the proof does not seem to be adaptable to linear combinations of eigenfunctions.
In our next result, we prove lower bounds for linear combinations of eigenfunctions which we call {\it reverse Bernstein inequalities}.
We mention in passing that our results solve a question raised in \cite{shi2010} about extending \eqref{sx10} for the Neumann Laplacian on compact manifolds (see the end of the paper).  

\begin{theorem}\label{theo3-p}
Suppose that \eqref{doublin} and $(G)$ are satisfied and $L = -\Delta + W$ has discrete spectrum (with the above notations \eqref{dis}). Then we have the following  assertions. 
\begin{enumerate}[i) ]
\item For $q \in [2, \infty]$ there exists a constant  $C>0$ such that for every $N \in \N$ and any complex sequence $(\alpha_k)_{k\geq N}$ with finite support, the following reverse Bernstein inequality holds
\begin{equation*}
 C \lambda_N \, 
\|  \sum_{k=N}^{+\infty}  \alpha_k \phi_k   \|_q\leq \| \nabla \Big{(} \sum_{k=N}^{+\infty}  \alpha_k \phi_k \Big{)} \|_q  + \| \sqrt{W} \Big{(}  \sum_{k=N}^{+\infty}  \alpha_k \phi_k \Big{)}  \|_q. \eqno(RB_q)
\end{equation*}
\item For $q\in (1,2)$, if the Bernstein inequality $(B_p)$ holds for $p=\frac{q}{q-1}$, then the reserve inequality $(RB_q)$ holds.
\item For $q=1$, if  $(B_\infty)$ and $\inf\limits_{x\in M}V(x,1)>0$ are satisfied, then $(RB_1)$ holds.
\end{enumerate}
\end{theorem}

To the best of our knowledge, the previous \textit{reverse Bernstein inequalities} are new even for trigonometric polynomials. In this case, the action of the gradient is similar to that of a multiplier and $(RB_q)$ holds for all $q\in [1,+\infty]$. Therefore,  for any $N\in \N^\star$ and any sequence $(\alpha_k)_{|k|\geq N}$ with finite support, one has
\begin{equation}\label{fourier}
\big\Vert \sum_{\substack{ k\in \Z\\ |k|\geq N}} \alpha_k e^{ik .} \big\Vert_{L^q(-\pi,\pi)}\leq \frac{C}{N}
\big\Vert \sum_{\substack{k\in \Z \\ |k|\geq  N}} k \alpha_k e^{ik .} \big\Vert_{L^q(-\pi,\pi)}.
\end{equation}

The proof of the previous theorem appeal to techniques from harmonic analysis which we summarize as follows. 
\begin{itemize}
\item The reverse Bernstein inequality $(RB_q)$ will be a consequence of the following semi-classical inequality
\begin{equation*}
 \frac{C}{ \sqrt{h}} \|  \Psi(hL) u \|_{q}\leq \| \nabla  u \|_{q} +  \| \sqrt{W}  u  \|_{q}   \eqno(SRB_q)
 \end{equation*}
for $\Psi \in \CC^\infty(0, \infty)$, $ \Psi \equiv 0$ near $0$ and $\Psi \equiv 1$ near $\infty$. 
\item We suitably apply the weak factorization \eqref{dixm} in order to reduce $(SRB_q)$ to the case where $\Psi$ has compact support in $(0,+\infty)$. 
\item For $q<+\infty$,  $(SRB_q)$ with  $\Psi \in C_c^\infty(0, \infty)$ will follow from a duality argument similar to that used by Bakry for the Riesz transform (see \cite[Proposition 2.1]{coulhon03} and \cite[Section 4]{bakry}).
\item The case $q=+\infty$ needs more work and will use in an essential way 
that the Riesz transform type operators $\nabla L^{-1/2}$ and $ \sqrt{W}L^{-1/2}$ are bounded from $H^1_L(M)$  to $L^1(M)$, where $H^1_L(M)$ is the Hardy space associated with $L$ (see Hofmann et al. \cite{HLMMY} and the references there).
Such a boundedness result is proved by  Duong, Ouhabaz and Yan \cite{DOY2006} in the case  $M= \R^n$ but their work extends to manifolds under the assumptions \eqref{doublin} and $(G)$.
Bakry's duality argument will be then used to obtain the inequality
\begin{equation}\label{bakry-int}
C \| \sqrt{L} u \|_{{BMO_L(M)}}\leq  \| \nabla u \|_\infty + \| \sqrt{W} u \|_\infty,
\end{equation}
where $BMO_L(M)$ is a $BMO$ space associated with $L$.
\item   In order to obtain $(SRB_\infty)$ from \eqref{bakry-int} we prove for appropriate functions $\Psi$
\begin{equation*}
\frac{C}{\sqrt{h}} \| \Psi(hL)u\|_{\infty} \leq \| \sqrt{L} u \|_{{BMO_L(M)}}.
\end{equation*}
This  inequality expresses the uniform boundedness of $\Phi(hL)$ from $BMO_L(M)$ into $L^\infty(M)$ where $\Phi(x) = \frac{\Psi(x)}{\sqrt{x}}$. Actually, 
such a boundedness is surprising since one usually proves  boundedness from $L^\infty$ into $BMO$ but not the converse. See Theorem \ref{bmobmo} and Lemma \ref{sqrtL} for more details. 
%The point here is that $\Phi$ belongs to a specific class of functions including $\CC_c^\infty(0,+\infty)$ (a particular result was obtained in \cite[Theorems 9.1 % and 9.3]{imek-aif} via the Littlewood-Paley theory on boundaryless compact manifolds). 
%We will reach a general class of functions, including $\Phi$, with appropriate decay at $0$ and $+\infty$, by a use of the Dixmier-Malliavin weak factorization \eqref{dixm}.
\end{itemize}

\bigskip
\noindent{\bf Notation.} Throughout this paper,  the Lebesgue spaces $L^p(M)$ are considered with respect to the Riemannian measure $\mu$. The norm in $L^p(M)$ is denoted by $\| . \|_p$. For a given bounded operator $T : L^p(M) \to L^q(M)$ we write $\| T \|_{p\to q}$ to denote its norm. 
The duality $L^p(M)-L^{p'}(M)$ is denoted $(f,g)$, where $p'$ is the dual exponent of $p$. 
For a smooth function $f$, we  write $\| \nabla f \|_p$ instead of $\|  |\nabla f | \|_p$.\\
We shall often use $C$ to denote  a positive constant which may vary from line to line. 

\bigskip
\noindent{\bf Acknowledgments.} The authors would like to thank Peng Chen and Lixin Yan for  their generous  help for the proof of the boundedness of $\psi(L)$ from $L^1(M)$ into $H^1_L(M)$  for $\psi \in \CC_c^\infty(0, \infty)$.\\
The research of R. Imekraz is partly supported by the ANR project ESSED ANR-18-CE40-0028.
The research of E.M. Ouhabaz is partly supported by the ANR project RAGE ANR-18-CE40-0012-01.

%%%%.   %%%%%.  %%%%%%% 

\section{Semi-classical Bernstein inequalities}\label{semi}
In this section, we reformulate the Bernstein inequality $(B_p)$ as a semi-classical inequality.  As mentioned in the introduction, this reformulation does not need $L$ to have a discrete spectrum.  
\begin{theorem}\label{theo1-psc} Assume the doubling condition \eqref{doublin} and the Gaussian upper bound $(G)$. For any $\psi \in \CC_c^\infty([0, \infty))$, there exists a positive constant $C = C_\psi$ such that for any $p\in [1,2]$ the following holds
 \begin{equation*}
\left\Vert \nabla \psi(hL) f \right\Vert_{p} + \left\Vert \sqrt{W} \psi(hL) f \right\Vert_{p} \leq  \frac{C}{\sqrt{h}} \left\Vert f\right\Vert_{p}, \qquad h>0,\quad f\in L^p(M). \eqno(SB_p)
 \end{equation*}
\end{theorem}
Before proving Theorem \ref{theo1-psc}, we state for a pedagogical reason a version   for any $p\in [1,+\infty]$ which shows that $(SB_p)$ is equivalent to the regularity property $(R_p)$. Its  proof  will be given in Section \ref{sec-Riesz}. 

\begin{theorem}\label{theo2-psc}
Suppose \eqref{doublin} and  $(G)$. The following statements are equivalent for any $p\in [1,+\infty]$.
\begin{enumerate}[i) ]
\item There exists a non-zero function $\psi_0 \in \CC_c^\infty([0, \infty))$ for which the  semi-classical Bernstein inequality $(SB_{p})$ holds, i.e., 
\begin{equation*} \left\Vert \nabla \psi_0(h L) \right\Vert_{p \to p}+\left\Vert \sqrt{W} \psi_0(h L) \right\Vert_{p \to p}  \leq  \frac{C}{\sqrt{h}} \qquad \mbox{for all }h>0.
\end{equation*}
\item For any $\psi\in \CC_c^\infty([0, \infty))$, the semi-classical Bernstein inequality $(SB_{p})$ holds, i.e., 
\begin{equation*}
\left\Vert \nabla \psi(hL) \right\Vert_{p \to p}+\left\Vert \sqrt{W} \psi(hL) \right\Vert_{p \to p}  \leq  \frac{C'}{\sqrt{h}}  \qquad \mbox{for all }h>0.
\end{equation*}
\item The regularity property $(R_p)$ holds, i.e., 
\begin{equation*}\left\Vert \nabla e^{-hL}\right\Vert_{p \to p} + \left\Vert \sqrt{W} e^{-hL}\right\Vert_{p \to p}\leq\frac{C''}{\sqrt{h}}  \qquad \mbox{for all }h>0. \eqno{(R_p)}
 \end{equation*}
\end{enumerate}
\end{theorem}

It is now worthwhile to recall the connection of the regularity $(R_p)$ and the Riesz transform. Firstly, the boundedness of the Riesz transforms $\nabla L^{-1/2}$ and $\sqrt{W}L^{-1/2}$ on $L^p(M)$, for $p\in (1,+\infty)$,  implies $(R_p)$ as follows 
\begin{equation}\label{riesz-cons}
 \| \nabla e^{-tL}  \|_{p\rightarrow p}  + \| \sqrt{W} e^{-tL}  \|_{p\rightarrow p}  \leq   C \| L^{1/2} e^{-t L }  \|_{p\rightarrow p}  \leq  \frac{C}{\sqrt{t}} ,
\end{equation}
where in the last inequality we used the analyticity of the semigroup on $L^p(M)$. 

For the other side of the picture let us consider the case where  $W = 0$ and the  heat kernel of $e^{t\Delta}$ satisfies the full Li-Yau  estimates 
\begin{equation}\label{liyau}
\frac{c}{V(y,\sqrt{t})} \exp\left(-C\frac{d(x,y)^2}{t} \right)  \leq p_t(x,y)\leq  \frac{C}{V(y,\sqrt{t})} \exp\left(-c\frac{d(x,y)^2}{t} \right).
\end{equation}
It is proved by  Auscher et al. \cite{auscher2004} (see also Bernicot and Frey \cite{BF16} for some extensions) that if $(R_q)$ holds  for some $q > 2$ then  the Riesz transform $\nabla (-\Delta)^{-1/2}$ is bounded on $L^p$ for all $p \in [2, q)$. As a consequence, manifolds for which the  Riesz transform is not bounded on $L^p$ are counterexamples for the semi-classical Bernstein inequalities. 
For example, conical manifolds studied by Li in \cite{li-conic} \footnote{ for which \eqref{liyau} is true but  the Riesz transform is unbounded for some $p\in (2,+\infty)$, see \cite{coulhon-conic}.} are counterexamples for $(SB_p)$ for $p > 2$. 
%We finish that remark by underlining that Theorem \ref{theo2-psc}, as well as the analysis of \cite{auscher2004}, is applicable to conic manifolds (which are not %complete).

\begin{proof}[Proof of Theorem \ref{theo1-psc}] 
By interpolation, $(SB_1)$ and $(SB_2)$ will imply $(SB_p)$ for $p \in (1, 2)$.

$\bullet$ We start with the simple  case $p=2$. Using the fact that the quadratic form of $L$ satisfies 
$$ (\sqrt{L} u , \sqrt{L}u) = \int_M | \nabla u |^2 d \mu + \int_M W |u|^2 d\mu \qquad \mbox{for all } u \in D(\sqrt{L}) $$
we obtain for $f \in L^2(M)$ 
\begin{equation*}
\left\Vert \nabla \psi(hL) f \right\Vert_{2}^2 + \left\Vert \sqrt{W} \psi(hL) f \right\Vert_{2}^2  = \left\Vert \sqrt{L} \psi(hL) f \right\Vert_{2}^2=\frac{1}{{h}} \left\Vert \sqrt{hL} \psi(hL) f \right\Vert_{2}^2.
\end{equation*}
By setting $\Psi(x)=\sqrt{x}\psi(x)$ and using the standard functional calculus for the self-adjoint operator $L$ on $L^2(M)$, we obtain  
\begin{equation*}
\left\Vert \nabla \psi(hL) f \right\Vert_{2}^2 + \left\Vert \sqrt{W} \psi(hL) f \right\Vert_{2}^2  \leq \frac{\| \Psi \|^2_\infty}{h}\left\Vert f\right\Vert_{2}^2.
\end{equation*}
This gives $(SB_2)$. 

$\bullet$ Next we prove $(SB_1)$. As mentioned in the introduction we use some ideas which already appeared in the proofs of spectral multiplier theorems (cf. \cite{DOS}).  Given 
$\psi \in \CC_c^\infty([0, \infty))$ and define $\psi_e (x) := \psi(x) e^{2x}$. We can extend $\psi$ as a $\CC_c^\infty$-function on $\R$  and denote by 
$\widehat{\psi_e}$ its Fourier transform. 
%As noted in the introduction, one may notice the relation $\widehat{\psi_e}(\xi)=\widehat{\psi}(\xi+2i)$ which has a sense since $\widehat{\psi}$ can be %extended as a holomorphic function in $\C$. However, the main thing we need in the sequel is that $\widehat{\psi_e}$ belongs to the Schwartz space. 
The Fourier inverse formula then allows to write 
\begin{equation*}
\psi(x)  =  \int_\R e^{-2x} \widehat{\psi_e}(\xi) e^{i x.\xi} \frac{d\xi}{\sqrt{2\pi}}.
\end{equation*}
Therefore
\begin{equation}\label{eq0}
\psi(hL)  = \int_{\R} \widehat{\psi_e}(\xi)  e^{-(2- i\xi) h L} \frac{d\xi}{\sqrt{2\pi}}.
\end{equation}
Thus,
\begin{equation}\label{eq1}
\nabla \psi(hL)  = \int_{\R} \widehat{\psi_e}(\xi)  \nabla e^{-hL} e^{-(1- i\xi) h L} d\xi,  \quad  \sqrt{W} \psi(hL)  = \int_{\R} \widehat{\psi_e}(\xi)  \sqrt{W} e^{-hL} e^{-(1- i\xi) h L} d\xi.
\end{equation}
Now we estimate  
$ \| \nabla e^{-hL} e^{-(1- i\xi) h L} \|_{1 \to 1}$ and  $ \| \sqrt{W} e^{-hL} e^{-(1- i\xi) h L} \|_{1 \to 1}$. 

Following an argument of Grigor'yan one proves that there exists a constant $c_0 > 0$ such that
\begin{equation}\label{grigo}
\int_M  | \nabla_x p_h(x,y) |^2 \exp\left(c_0\frac{d(x,y)^2}{h} \right) d\mu(x) \leq \frac{C}{h V(y,\sqrt{h})}.
\end{equation}
See \cite{grigo95} or \cite[Lemma 2.3]{coul99} for $L = -\Delta$ on manifolds,  \cite[p. Section 6.6] {maati} for elliptic operators on domains and \cite[Proposition 3.1]{DOY2006} for Schr\"odinger operators on $\R^n$. The proof given in these papers is based on integration by parts and it carries overs  to $ L = -\Delta + W$ on manifolds. 
On the other hand, from the doubling  condition one has easily
%From the proof of \cite[Proposition 7.1]{maati} or \cite[Theorem 4.3]{maati2}, we know that for any $c>0$, any $T>0$ and for almost every $y\in X$, the following inequality yields
\begin{equation}\label{unifo}
\int_{M} \exp\left(-c\frac{d(x,y)^2}{h} \right)d\mu(x)\leq C V(y,\sqrt{h})
\end{equation}
for all $h > 0$ and $y \in M$ (see e.g., \cite[Proposition 7.1]{maati} or \cite[Theorem 4.3]{maati2}). Hence, by the Cauchy-Schwarz inequality
\begin{eqnarray*}
\int_M |  \nabla_x p_{h}(x,y) | d\mu(x) &=& \int_M  | \nabla_x p_{h}(x,y)| \exp\left(\frac{c_0d(x,y)^2}{h} \right) 
\exp\left(-\frac{c_0d(x,y)^2}{h} \right) d\mu(x) \\
& \leq & \frac{C}{\sqrt{h}}. 
\end{eqnarray*}
This shows  that $\nabla e^{-hL}$ is bounded on $L^1(M)$ and 
\begin{equation}\label{bound-L1}
\| \nabla e^{-hL} \|_{1 \to 1} \le \frac{C}{\sqrt{h}}.
\end{equation}
The inequality \eqref{grigo} still holds when $\nabla$ is replaced by multiplication by $\sqrt{W}$ (see \cite{DOY2006}).
We can then argue as previously and obtain 
\begin{equation}\label{bound-L1-1}
\| \sqrt{W} e^{-hL} \|_{1 \to 1} \le \frac{C}{\sqrt{h}}.
\end{equation}
In order to continue, we recall the following bound (see \cite[Theorem 4.3]{maati2})
\begin{equation}\label{holog-L1}
\| e^{-z L} \|_{q \to q} \leq C_\varepsilon \left(\frac{|z|}{\Re(z)} \right)^{n\left\vert\frac{1}{2}-\frac{1}{q}\right\vert + \varepsilon},
\end{equation}
for any $\varepsilon > 0$ and  all $z\in \C^+$ and  $q\in [1,+\infty]$. 
Using \eqref{holog-L1} for $q=1$ it follows that 
\begin{eqnarray*}
\| \nabla \psi (hL) \|_{1 \to 1}  
&\le& C_\varepsilon\| \nabla e^{-hL} \|_{1 \to 1}  \int_{\R} | \widehat{\psi_e}(\xi) |   (1 + \xi^2)^{\frac{n}{4} + \varepsilon}, \\
\| \sqrt{W} \psi (hL) \|_{1 \to 1} & \leq & C_\varepsilon\| \sqrt{W} e^{-hL} \|_{1 \to 1}  \int_{\R} | \widehat{\psi_e}(\xi) |   (1 + \xi^2)^{\frac{n}{4} + \varepsilon}.
\end{eqnarray*}
From \eqref{bound-L1} and \eqref{bound-L1-1} we obtain
\begin{eqnarray*}
\| \nabla \psi (hL) \|_{1 \to 1} + \| \sqrt{W} \psi (hL) \|_{1 \to 1} 
&\le& \frac{C'_\varepsilon}{\sqrt{h}} \left(  \int_{\R} | \widehat{\psi_e}(\xi) |^2   (1 + \xi^2)^{\frac{n}{2} + 1+ 4 \varepsilon} d\xi \right)^{1/2}\\
&= & \frac{C'_\varepsilon}{\sqrt{h}}  \| \psi_e \|_{H^{n/2+1+4\varepsilon}(\R)} \\
&\le& \frac{C''_\varepsilon}{\sqrt{h}}  \| \psi \|_{H^{n/2+1+4\varepsilon}(\R)}. 
\end{eqnarray*}
This gives  $(SB_1)$.
\end{proof}
\begin{remark}\label{remrem} We mention  few other results that one can obtain using the previous proof. 
\begin{enumerate}[a) ] 
\item Similarly to \cite{DOS}, we note that \eqref{eq0} and \eqref{holog-L1} allow to recover the well-known fact : under the assumptions \eqref{doublin} and $(G)$, for any $\psi\in \CC_c^\infty([0,+\infty))$ and any $q\in [1,+\infty]$ the following inequality holds true
\begin{equation}\label{mult}
\sup\limits_{h>0}\left\Vert \psi(hL) \right\Vert_{q\rightarrow q}<+\infty.
\end{equation}
\item[ ] Pseudo-differential proofs of \eqref{mult} exist but need to consider a specific pseudo-differential framework (see \cite[Cor 2.2]{burq2004} and \cite[Theorem D.1]{imek-aif}).
\item By writing $z=\frac{\Re(z)}{2}+\left(\frac{\Re(z)}{2}+i\Im(z) \right)$, we easily deduce from \eqref{bound-L1}, \eqref{bound-L1-1} and \eqref{holog-L1} the inequality 
\begin{equation*}
\forall z\in \C^+ \quad \left\Vert \nabla e^{-zL}\right\Vert_{1\rightarrow 1} +\left\Vert \sqrt{W}e^{-zL}\right\Vert_{1\rightarrow 1}\leq \frac{C_\varepsilon}{\sqrt{\Re(z)}} \left( \frac{|z|}{\Re(z)}\right)^{\frac{n}{2}+\varepsilon},
\end{equation*}
which in turn could directly be used to handle \eqref{eq1}.
\item The previous proof  shows the implication iii) $\Rightarrow$ ii) of Theorem \ref{theo2-psc}.
\end{enumerate}
\end{remark}

In the next result we prove $L^p(M)-L^q(M)$ Bernstein inequalities.

\begin{theorem}\label{theo-LpLqsc} As above, we assume the doubling condition \eqref{doublin}.
We  assume that the heat kernel satisfies the Gaussian bound $(G)$ and there exists $m > 0$ such that 
 \begin{equation} \label{infty-encore}
p_t(x,x)  \leq  \frac{C}{t^{m/2}}, \quad t\in (0,1] \quad \mbox{a.e.  }x \in M.
 \end{equation}
Then for $\psi \in \CC_c^\infty([0, \infty))$, there exists a positive constant $C $ such that
for any  couple $(p,q)$ satisfying $1 \le p \le q \le 2 $, we have
 \begin{equation*}
\left\Vert \nabla \psi(hL) f \right\Vert_{q} + \left\Vert \sqrt{W} \psi(hL) f \right\Vert_{q} \leq C h^{-\left(\frac{1}{2} + \frac{m}{2} |\frac{1}{p} - \frac{1}{q} |\right) } \left\Vert f\right\Vert_{p},\quad h \in (0,1],\quad f\in L^p(M). \eqno(SB_{p,q})
 \end{equation*}
If  the regularity property  $(R_{q_0})$ is satisfied for some $q_0 > 2$, then $(SB_{p,q})$ holds for all $1\leq p\leq q \leq q_0$. 
\end{theorem}
\begin{proof}
We shall use the same strategy as in Theorem \ref{theo1-psc}.  Firstly, we prove the estimate
\begin{equation}\label{aprouv}\| \nabla e^{-hL}  \|_{p \to q} + \| \sqrt{W}  e^{-hL}  \|_{p \to q} \leq Ch^{-\frac{1}{2} - \frac{m}{2} |\frac{1}{p} - \frac{1}{q} |}, \qquad h\in (0,1].
\end{equation}
We recall a classical fact that the semigroup identity and the symmetry of the heat kernel imply 
\begin{eqnarray*}
|p_h(x,y)|^2 & =& \left\vert \int_{M} p_{h/2}(x,z) p_{h/2}(z,y) d\mu(z)     \right\vert^2  \\
& \leq & \int_{M} p_{h/2}(x,z)^2 d\mu(z) \int_{M} p_{h/2}(z,y)^2 d\mu(z) = p_{h}(x,x)p_{h}(y,y) \le \frac{C}{h^{m}}.
\end{eqnarray*}
This gives 
\begin{equation*}
\| e^{-hL} \|_{1 \to \infty} \le  \frac{C}{h^{m/2}}, \ h \in (0, 1]
\end{equation*}
and hence by interpolation,  
\begin{equation*}
\| e^{-h L} \|_{p \to q} \le C h^{-\frac{m}{2} \left\vert \frac{1}{p}-\frac{1}{q}\right\vert      }.
\end{equation*}
%Moreover, by the semigroup property, we can write
%$$\nabla e^{-hL} = \nabla e^{-\frac{h}{2} L}  e^{-\frac{h}{2} L} $$  
Thus,
\begin{eqnarray*}
\| \nabla e^{-hL}  \|_{p \to q} + \| \sqrt{W}  e^{-hL}  \|_{p \to q} &\le& \left(  \| \nabla e^{-\frac{h}{2} L}   \|_{q \to q} + \| \sqrt{W}  e^{-\frac{h}{2}L}  \|_{q \to q}\right) \| e^{- \frac{h}{2}L}  \|_{p \to q}\\
&\le&  C\left(  \| \nabla e^{-\frac{h}{2} L}   \|_{q \to q} + \| \sqrt{W}  e^{-\frac{h}{2}L}  \|_{q \to q}\right)h^{-\frac{m}{2}\left\vert \frac{1}{p}-\frac{1}{q}\right\vert}.
\end{eqnarray*}
For $2 \le p \le q \le q_0$, the conclusion comes easily from the additional assumption $(R_{q_0})$.\\
Using \eqref{eq1} and \eqref{holog-L1}, we obtain
\begin{equation*}
\|\nabla \psi(hL)\|_{p\to q}+
\|\sqrt{W} \psi(hL)\|_{p\to q} \leq (\|\nabla e^{-hL}\|_{p\to q}+
\|\sqrt{W} e^{-hL}\|_{p\to q})\int_{\R} |\widehat{\psi_e}(\xi)|(1+\xi^2)^{\frac{n}{2}\left\vert\frac{1}{2}-\frac{1}{q} \right\vert +\frac{\varepsilon}{2}}d\xi.
\end{equation*}
We then argue as in  Theorem \ref{theo1-psc}.
\end{proof}

\section{Discrete and semi-classical Bernstein inequalities} \label{derivation}

In this section, we assume the doubling condition \eqref{doublin}, the Gaussian bound $(G)$ and that $L$ has discrete spectrum with the notation of  \eqref{dis}. We show the equivalence of the discrete Bernstein inequality $(B_p)$ stated in Theorems \ref{theo1-p} and \ref{theo2-p} and the semi-classical Bernstein inequality $(SB_p)$ stated in Theorems \ref{theo1-psc} and \ref{theo2-psc}.
Since Theorem \ref{theo1-psc} has been proved in the previous section, we obtain Theorem \ref{theo1-p}. We also derive Theorem \ref{theo-LpLq} from Theorem \ref{theo-LpLqsc}. Similarly, the present section shows that Theorem \ref{theo2-psc} implies  Theorem \ref{theo2-p} (we recall that Theorem \ref{theo2-psc} will be proved in Section \ref{sec-Riesz}).

\begin{proposition}\label{proequiv}
Let $p \in [1, \infty]$. Then the discrete Bernstein inequality $(B_p)$ and the semi-classical one $(SB_p)$ (with some non-trivial $\psi \in \CC_c^\infty([0, \infty))$) are equivalent.
\end{proposition}
\begin{proof} We start with $(SB_p) \Rightarrow (B_p)$. By Theorem \ref{theo2-psc}  if  $(SB_p)$ is satisfied by one non-trivial function in $ \psi \in \CC_c^\infty([0, \infty))$ then it holds for any other function in the same space. Now we choose $\psi$ such that 
\begin{equation}\label{defipsi}
\psi(x)=\left\{ \begin{array}{rcl}  1 & \mbox{for} & x\in [0,1]\\
0 & \mbox{for} & x\geq 2.      \end{array}   \right.
\end{equation}
Set  $f=\sum\limits_{k=0}^{N} \alpha_k \phi_k$ with $(\alpha_0,\dots,\alpha_N)\in \C^{N+1}$ and $N \ge 1$.  We   
take   $h = \frac{1}{\lambda^2_N}$ and write
\begin{equation*}
\psi(hL) f = \sum_{k=0}^N \alpha_k \psi(hL) \phi_k= \sum_{k=0}^N \alpha_k \psi\left(\frac{ \lambda^2_k}{\lambda^2_N}\right) \phi_k= \sum_{k=0}^N \alpha_k  \phi_k.
\end{equation*}
We apply  $(SB_p)$ and obtain immediately 
\begin{eqnarray*}
 % \| \nabla \psi(hL)f \|_{p}+\|\sqrt{W} \psi(hL)f\|_{p}       & \leq & \frac{C}{\sqrt{h}} \|f\|_{p}\\
 \| \nabla \Big{(}  \sum_{k=0}^N  \alpha_k \phi_k \Big{)}  \|_p  + \| \sqrt{W} \Big{(}  \sum_{k=0}^N  \alpha_k \phi_k \Big{)}  \|_p & \le & C \lambda_N \, 
\|  \Big{(}  \sum_{k=0}^N  \alpha_k \phi_k \Big{)}  \|_p.
\end{eqnarray*}
This proves $(B_p)$.

We prove the converse $(B_p) \Rightarrow (SB_p)$.  
We merely consider the case $f\in L^p(M) \cap L^2(M)$, then the general case $f\in L^p(M)$ follows by density. The definition \eqref{defipsi} of $\psi$ ensures that $\psi(hL)f$ is spectrally localized on $[0,2/h]$ with respect to $L$. So there are $N\in \N^\star$ and coefficients $\alpha_0,\dots,\alpha_N$ such that 
\begin{equation*}
\psi(hL) f = \sum_{k=0}^N \alpha_k \phi_k \qquad \mbox{with}\quad \lambda_N^2 \leq \frac{2}{h}.
\end{equation*}
Applying $(B_p)$ yields, 
\begin{equation*}
 \| \nabla \psi(hL)f \|_{p}+\|\sqrt{W} \psi(hL)f\|_{p}\leq C\lambda_N \|\psi(hL) f\|_{p}\leq \frac{C'}{\sqrt{h}} \|\psi(hL) f\|_{p}.
\end{equation*}
Since the multiplier $\psi(hL)$ is  bounded uniformly in $h > 0$ on  $L^p(M)$ (see \eqref{mult} or \cite{DOS}) it follows that  $\|\psi(hL) f\|_{p}\leq C \|f\|_{p}$.  This gives  the semi-classical inequality $(SB_p)$ with $\psi$ as in \eqref{defipsi}. The same estimate holds for  all $\psi \in \CC_c^\infty([0, \infty))$ by Theorem \ref{theo2-psc}.

\end{proof}

\begin{proof}[Proof of Theorem \ref{theo-LpLq}] By  Theorem \ref{theo-LpLqsc} we have for every $\psi \in \CC_c^\infty([0, \infty))$
\begin{equation*}
\left\Vert \nabla \psi(hL) f \right\Vert_{q} + \left\Vert \sqrt{W} \psi(hL) f \right\Vert_{q} \leq C h^{-\left(\frac{1}{2} + \frac{m}{2} |\frac{1}{p} - \frac{1}{q} |\right) } \left\Vert f\right\Vert_{p}, \quad h\in (0,1],\quad f\in L^p(M) \eqno(SB_{p,q})
\end{equation*}
As in the previous proof, we chose $\psi$ as in \eqref{defipsi} and we  take  $h=\frac{1}{\lambda_N^2}$ and $f=\sum\limits_{k=0}^N \alpha_k \phi_k$ to obtain 
$(B_{p,q})$ for $\lambda_N \ge 1$. 
For small $\lambda_N$, we argue as follows. 
%%%%%%%%%
Let $N_0$ be the smallest  integer such that  $\lambda_{N_0}\geq 1$ and let  $ 1 \le N < N_0$. We first have the rough bound
\begin{equation}\label{CN1}
\|\nabla \sum_{k=0}^N \alpha_k \phi_k\|_{q}+\|\sqrt{W}\sum_{k=0}^N \alpha_k \phi_k \|_q \leq C_{N,1}\sup\limits_{0\leq k \leq N} |\alpha_k| ,
\end{equation}
with $C_{N,1}:=\sum\limits_{k=0}^N \|\nabla \phi_k\|_q+\|\sqrt{W}\phi_k\|_q $.
By  the equivalence of the following two norms on $\C^{N+1}$
\begin{equation*}
(\alpha_0,\dots,\alpha_N)\mapsto         \sup\limits_{0\leq k \leq N} |\alpha_k|  \qquad \mbox{et}\qquad    
(\alpha_0,\dots,\alpha_N)\mapsto     \|\sum_{k=0}^N \alpha_k \phi_k \|_{p}      \end{equation*}
one has  for a suitable constant $C_{N,2}>0$ the following inequality
\begin{equation*}
\|\nabla \sum_{k=0}^N \alpha_k \phi_k\|_{q}+\|\sqrt{W}\sum_{k=0}^N \alpha_k \phi_k \|_q \leq C_{N,2}\|\sum_{k=0}^N \alpha_k \phi_k \|_{p}.
\end{equation*}
Since $\lambda_N\geq \lambda_1>0$ we may set
\begin{equation*}
C:=\sup\limits_{1\leq N<N_0}\frac{C_{N,2}}{\lambda_N^{1+m\left\vert\frac{1}{p}-\frac{1}{q} \right\vert} }
\end{equation*}
so that we have 
\begin{equation*}
\|\nabla \sum_{k=0}^N \alpha_k \phi_k\|_{q}+\|\sqrt{W}\sum_{k=0}^N \alpha_k \phi_k \|_q \leq C\lambda_N^{1+m\left\vert\frac{1}{p}-\frac{1}{q} \right\vert}\|\sum_{k=0}^N \alpha_k \phi_k \|_{p}.
\end{equation*}
This proves $(B_{p,q})$. 
\end{proof}

\section{Dixmier-Malliavin weak factorization}\label{dixmier-M}

 This section is devoted to state a result which is in the spirit of the paper \cite{dixmi} by Dixmier and Malliavin.
More precisely, \cite{dixmi} studies the possibility of decomposing a function  in a Fr\'echet functional space into a finite sum of convolutions under the action of a Lie group $G$.
The statements there are however written in the language of the theory of representations on Lie groups whereas we are interested in the particular case $G=\R$, only. Hence
instead of using the whole machinery of the paper \cite{dixmi} (more precisely its Theorem 3.3), we give a relatively simpler and direct proof
based on  Lemma 2.5 from \cite{dixmi}. Our proof gives  two convolutions in the factoraization. 
\begin{lemma}\label{dix}[Dixmier-Malliavin]
For any positive sequence $(\beta_n)_{n\in \N}$, there exist a positive sequence $(\alpha_n)_{n\in \N}$ and two functions $\rho_1$ and $\rho_2$ belonging to $\CC_c^\infty(\R)$ satisfying
\begin{enumerate}[i) ]
\item $\alpha_n\leq \beta_n$ pour tout $n\geq 1$
\item for any function $F\in \CC^\infty(\R)$ the following limit holds 
\begin{equation*}
\lim\limits_{p\rightarrow +\infty} \langle \rho_1 \star \sum_{n=0}^p (-1)^n \alpha_n \delta^{(2n)},F \rangle = F(0)+\int_{\R} F(x)\rho_2(x)dx.
\end{equation*}
\end{enumerate}
\end{lemma}

We then have the following decomposition lemma.

\begin{lemma}\label{dixm2}
Let $f\in L^1(\R)\cap \CC^\infty(\R)$ be a function whose all derivatives belong to $L^1(\R)$.
Then there exist $g\in L^1(\R)$, $\rho_1\in \CC_c^\infty(\R)$ and $\rho_2\in \CC_c^\infty(\R)$ such that the following (weak) factorization holds
\begin{equation*}
f=\rho_1\star f+\rho_2\star g.
\end{equation*}
\end{lemma}

We also state  the following reformulation which  will be used several times in this paper.

\begin{corollary}\label{dixm3}
Let $F:(0,+\infty)\rightarrow \R$ be a smooth function satisfying 
\begin{equation}\label{dixm3-h}
\int_{0}^{+\infty}x^{k-1} |F^{(k)}(x)| dx<+\infty \qquad \mbox{ for all } k\in \N.
\end{equation}
Then there exist  $\Theta_1, \Theta_2 \in \CC_c^\infty(0,+\infty)$ and  $F_1, F_2 \in L^1(0,+\infty)$ such that
\begin{equation}\label{dixm10}
F(x)=\int_{0}^{+\infty} \Theta_1\left( \frac{x}{y}\right) F_1(y)+\Theta_2\left( \frac{x}{y}\right) F_2(y)dy,\qquad x\in (0,+\infty).
\end{equation}
\end{corollary}
\begin{proof}
The function $f: t\in \R\mapsto F(e^t)$ is integrable on $\R$ and one checks by induction that $f^{(k)}$ is a linear combination of the integrable functions $t\in \R\mapsto e^{jt}F^{(j)}(e^t)$, for $j$ being an integer belonging to $[1,k]$.
We then apply Lemma \ref{dixm2} and set 
\begin{eqnarray*}
\Theta_1(x) :=\rho_1(\ln(x)), & \qquad &F_1(x) :=\frac{f(\ln(x))}{x},
 \\
\Theta_2(x) :=\rho_2(\ln(x)), & \qquad & F_2(x) :=\frac{g(\ln(x))}{x}
\end{eqnarray*}
to obtain the decomposition \eqref{dixm10} on $(0, +\infty)$. 
\end{proof}

\begin{proof}[Proof of Lemma \ref{dixm2}]
We apply the Lemma \ref{dix} for $\beta_n=\frac{1}{(1+n^2)(1+||f^{(2n)}||_{L^1(\R)})}$ so that we have
\begin{equation}\label{lli}
\sum_{n\in \N} \beta_n \| f^{(2n)}\|_{L^1(\R)}<+\infty.
\end{equation}
Let  $(\alpha_n)_{n\in \N}$ be the positive sequence given by Lemma \ref{dix}. We will prove the following limit in $L^1(\R)$ as 
$p$ tends to $+\infty$,
\begin{equation}\label{zop}          \rho_1 \star \left(\sum_{n=0}^p (-1)^n \alpha_n f^{(2n)}\right)   \rightarrow f+\rho_2\star f. 
\end{equation}
We first remark that the sum $\sum (-1)^n \alpha_n f^{(2n)}$ is absolutely convergent in $L^1(\R)$ thanks to \eqref{lli} and to the inequality $\alpha_n \leq \beta_n$ given by Lemma \ref{dix}.
Since $\rho_1$ belongs to $L^1(\R)$, the boundedness of the convolution product from $L^1(\R)\times L^1(\R)$ to $L^1(\R)$ implies that the left-hand side of \eqref{zop} converges in $L^1(\R)$ as $p\rightarrow +\infty$.

To prove that the limit of \eqref{zop} is indeed $f+\rho_2 \star f$ it is sufficient to check it in the weak sense. Let $h\in \CC_c^\infty(\R)$ be a test function and let us prove that 
\begin{equation}\label{lli2}     \lim\limits_{p\rightarrow +\infty} \int_{\R} h(x)\left(   \rho_1 \star \left(\sum_{n=0}^p (-1)^n \alpha_n f^{(2n)}\right)\right)(x) dx = \int_{\R}h(x)f(x)+h(x)(\rho_2\star f)(x)dx. 
\end{equation}
We now recall the following consequence of  Fubini's theorem 
\begin{equation}\label{Fubi}  \int_{\R}a(x) (b\star c)(x) dx = \int_{\R} (a\star \Check{b})(x) c(x)dx     
\end{equation}
for all $ a\in \CC_c^\infty(\R),  b\in \CC^\infty(\R)$ and $c\in \CC_c^\infty(\R)$.  Here we use the  convention $\Check{b}(x)=b(-x)$. We then  write
\begin{equation*} \rho_1 \star \left(\sum_{n=0}^p (-1)^n \alpha_n f^{(2n)}\right)= \rho_1 \star \left(\sum_{n=0}^p (-1)^n \alpha_n \delta^{(2n)}   \right)   \star f= \left(\sum_{n=0}^p (-1)^n \alpha_n \rho_1^{(2n)} \right)  \star f,\end{equation*}
which in turn gives, thanks to \eqref{Fubi},
\begin{equation*}
\int_{\R} h(x)\left(   \rho_1 \star \left(\sum_{n=0}^p (-1)^n \alpha_n f^{(2n)}\right)\right)(x) dx  =  
\int_{\R}\left( \sum_{n=0}^p (-1)^n \alpha_n \rho_1^{(2n)}(x) \right)(\Check{f}\star h)(x) dx.
\end{equation*}
We apply  Lemma \ref{dix} to the LHS and  obtain the following limit as $p\rightarrow +\infty$ 
\begin{equation*}
   (\Check{f}\star h)(0)+\int_{\R} (\Check{f}\star h)(x) \rho_2(x) dx,     \end{equation*}
By \eqref{Fubi} we rewrite  the previous limit as follows 
\begin{equation*}  \int_{\R}h(x)f(x)dx + \int_{\R}  h(x) (f\star \rho_2)(x) dx  .\end{equation*}
Hence the limit in \eqref{zop} is proved in $L^1(\R)$ and we have indeed proved  the following equality
%\begin{equation*}        f=-\rho_2\star f+\rho_1\star \underbrace{\left(\sum_{n\in \N} (-1)^n \alpha_n f^{(2n)} \right)}_{\in L^1}       \end{equation*}
\begin{equation*}        f=-\rho_2\star f+\rho_1\star \left(\sum_{n\in \N} (-1)^n \alpha_n f^{(2n)} \right)  \end{equation*}
with $\sum\limits_{n\in \N} (-1)^n \alpha_n f^{(2n)} \in L^1(\R)$. 
This proves the lemma.
\end{proof}

\section{Proof of Theorem \ref{theo2-psc}}\label{sec-Riesz}
This section is devoted to the proof of Theorem \ref{theo2-psc}. As we already mentioned in Remark \ref{remrem}, we have seen the implication iii) $\Rightarrow$ ii). The implication ii) $\Rightarrow$ i) is obvious. It  remains to prove the implication i) $\Rightarrow$ iii). We divide the proof into three steps. 

%\textbf{iii) $\Rightarrow$ ii)}. We write $\psi(x)=e^{-x}\widetilde{\psi}(x)$ with $\widetilde{\psi}(x)=e^{x}\psi(x)$.
%The $L^\infty$-functional calculus for the self-adjoint operator $L$ then gives for any $h>0$
%\begin{equation*}
%\psi(hL)= e^{-hL} \circ \widetilde{\psi}(hL).
%\end{equation*}
%Hence, we get
%\begin{eqnarray*}
%\|\nabla\psi(hL)\|_{p\rightarrow p}& =& \|\nabla e^{-hL}\|_{p\rightarrow p} \|\widetilde{\psi}(hL)\|_{p\rightarrow} 
%\leq \frac{C}{\sqrt{h}}, \\
%\|\sqrt{W}\psi(hL)\|_{p\rightarrow p}& =& \|\sqrt{W} e^{-hL}\|_{p\rightarrow p} \|\widetilde{\psi}(hL)\|_{p\rightarrow} 
%\leq \frac{C}{\sqrt{h}}.
%\end{eqnarray*}
%where we used that $\widetilde{\psi}(hL)$ is uniformly bounded from $L^p(M)$ to $L^p(M)$ with respect to $h>0$.

%\textbf{ii) $\Rightarrow$ i)}  is obvious.
%\textbf{i) $\Rightarrow$ iii)}.
\noindent  \textbf{Step 1.} Let $\psi_0$ be as in the theorem.  First, we observe that we can replace $\psi_0$ with $\psi_0^2$. Indeed,
\begin{eqnarray*}
 & & \| \nabla \psi^2_0 (hL) \|_{p\to p} + \| \sqrt{W} \psi^2_0 (hL) \|_{p\to p}  \\
& & \qquad \leq    \| \nabla \psi_0 (hL) \|_{p\to p} \|\psi_0(hL) \|_{p\to p}+\| \sqrt{W} \psi_0 (hL) \|_{p\to p}\|\psi_0(hL)\|_{p\to p}\\ 
& & \qquad \leq  \frac{C}{\sqrt{h}} \| \psi_0(hL) \|_{p\to p} \\
& & \qquad \leq  \frac{C}{\sqrt{h}}.
\end{eqnarray*}
By replacing $\psi_0$ by its square, we may assume that there exists $[a, b] \subset (0, \infty)$ such that 
\begin{equation}\label{eq-positivity}
 \psi_0 \ge 0 \ {\rm and} \ \psi_0(x)  > 0 \ {\rm for} \ x \in [a,b].
 \end{equation}
\textbf{Step 2.} Now let $\psi \in \CC_c^\infty(0, \infty)$ with support contained in $[c, d] \subset (0, \infty)$.  Set $K=\frac{b}{a}>1$ and consider an integer $N$ large enough so that $[c,d]\subset [K^{-N}a ,K^{N}b]$.
The equality $[K^{-N}a,K^N b]=\bigcup\limits_{-N\leq n\leq N} [K^n a,K^n b]$ and \eqref{eq-positivity} imply 
\begin{equation*}
  \sum_{n=-N}^N  \psi_0 (K^{-n} x) >0\quad \mbox{ for all }x \in [K^{-N}a,K^N b].
\end{equation*}
We now consider a function $\varphi \in \CC_c^\infty(\R)$ such that for any  $x \in [K^{-N} a, K^N b]$ the following holds
\begin{equation*}
\sum_{n=-N}^N \psi_0(K^{-n} x) \varphi (x) = 1.
\end{equation*}
Since the function $\sum\limits_{n=-N}^N  \psi_0 (K^{-n} \cdot)\varphi$ equals $1$ on the support of $\psi$ we 
have  $\psi=\sum\limits_{n=-N}^N \psi_0(K^{-n}\cdot)  \varphi\psi$. We argue exactly as in Step 1 to conclude that $\psi$ satisfies the semi-classical Bernstein inequality $(SB_p)$. %Note that the last argument gives a multiplicative loss due to the parameter $K$ but that does not matter since $\psi$ is a finite sum.

\textbf{Step 3.} What we proved in Step 2 is that we may replace a particular non-zero function $\psi_0\in \CC_c^\infty([0,+\infty))$ by any $\psi\in \CC_c^\infty(0,+\infty)$. In other words, we have obtained  ii) of Theorem \ref{theo2-psc} for the particular case of functions $\psi\in \CC_c^\infty(0,+\infty)$.
This particular case will be now combined with the weak factorization results of Section \ref{dixmier-M} to reach the heat propagator of assertion iii). 
We apply Corollary \ref{dixm3} to the function  $F(x)=\sqrt{x}e^{-x}$ for $x\in (0,+\infty)$. Then there exist
two functions $\Theta_1$ and $\Theta_2$ belonging to $\CC_c^\infty(0,+\infty)$ and two functions $F_1$ and $F_2$ belonging to $L^1(0,+\infty)$ so that 
\begin{equation*}
\sqrt{x} e^{-x}  = \int_{0}^{+\infty} \Theta_1\left( \frac{x}{y}\right) F_1(y) dy+
\int_{0}^{+\infty} \Theta_2\left( \frac{x}{y}\right) F_2(y) dy.
\end{equation*}
Let us now introduce the following two smooth functions $\psi_1$ and $\psi_2$ (they are compactly supported in $(0,+\infty)$)
\begin{equation*}
\psi_1(x) :=\frac{\Theta_1(x)}{\sqrt{x}}\qquad \mbox{and} \qquad\psi_2(x) :=\frac{\Theta_2(x)}{\sqrt{x}}.
\end{equation*}
Therefore, we have
\begin{equation*}
e^{-x} = \int_{0}^{+\infty} \psi_1\Big(\frac{x}{y}\Big)  F_1(y)\frac{dy}{\sqrt{y}}+
\int_{0}^{+\infty} \psi_2\Big(\frac{x}{y}\Big)  F_2(y)\frac{dy}{\sqrt{y}},
\end{equation*}
which leads to 
\begin{equation*}
e^{-hL} =  \int_{0}^{+\infty} \psi_1\left(\frac{hL}{y} \right)F_1(y) \frac{dy}{\sqrt{y}}+
\int_{0}^{+\infty} \psi_2\left(\frac{hL}{y} \right)F_2(y) \frac{dy}{\sqrt{y}}. 
\end{equation*}
Now we use Step 2 and obtain 
\begin{eqnarray*}
& & \left\Vert \nabla e^{-hL}\right\Vert_{p\to p}+\left\Vert \sqrt{W} e^{-hL}\right\Vert_{p\to p}\\
& & \qquad \qquad \leq \sum_{i=1}^2 \int_{0}^{+\infty} \Big(\left\Vert 
\nabla\psi_i\left(\frac{hL}{y} \right)\right\Vert_{p \to p}+ \left\Vert 
\sqrt{W}\psi_i\left(\frac{hL}{y} \right)\right\Vert_{p \to p}\Big)  \frac{|F_i(y)|}{\sqrt{y}} dy \\
&  & \qquad \qquad \leq \frac{C}{\sqrt{h}} \int_{0}^{+\infty} |F_1(y)|+|F_2(y)|dy = \frac{C'}{\sqrt{h}}.
\end{eqnarray*}
This proves the  regularity property $(R_p)$ of assertion iii).

 \section{A multiplier theorem from $BMO_L(M)$ into $L^\infty(M)$}\label{hardy}

Recall that $M$ is a Riemannian manifold satisfying the doubling property and $L = -\Delta + W$ is a Schr\"odinger operator with a non-negative potential $W \in L_{loc}^1(M)$.
Moreover the heat kernel of $L$ is assumed to satisfy the Gaussian upper bound $(G)$.
There is a large literature on  suitable Hardy or BMO spaces associated to $L$.
In the present paper, we use the general framework developed in \cite{HLMMY} (whose introduction contains references of important prior works). In \cite{HLMMY} it is required that
\begin{enumerate}[i) ]
\item $M$ is a metric measured space,
\item the measure $\mu$ is doubling,
\item the operator $L$ generates an analytic semigroup $(e^{-tL})_{t>0}$ satisfying the so-called Davies-Gaffney condition.
\end{enumerate}

Under these general assumptions, one can define a Hardy space $H^1_L(M)$ (see \eqref{e2.3}) and a $BMO_L(M)$ space associated with $L$. Note that $BMO_L(M)$ is the dual space of $H^1_L(M)$ (see \cite[Theorem 2.7]{HLMMY}). Recall that these spaces coincide with the usual Hardy $H^1(\R^d)$ and $BMO(\R^d)$ spaces in the Euclidean setting, i.e.,  $L=-\Delta$ and $M=\R^d$. 

Most of operators associated with $L$ such as the functional calculus or the  Riesz transform are  bounded from $H_L^1(M)$ to $L^1(M)$.  By duality one has boundedness results from $L^\infty(M)$ to $BMO_L(M)$. 
In the main result of the present section,  we prove boundedness of the functional calculus of $L$ (for a class of functions)  from  $L^1(M)$ into $H^1_L(M)$ or from $BMO_L(M)$ into $L^\infty(M)$. More precisely we have

\begin{theorem}\label{bmobmo} Let  $\varphi\in \CC^\infty(0,+\infty)$ such that 
\begin{equation}\label{bmocond}
 \int_{0}^{+\infty} x^{k-1} |\varphi^{(k)}(x)| dx <+\infty \quad \mbox{for all }k\in \N.
\end{equation}
Then we have
\begin{eqnarray}\label{bmobmo-a}
\sup_{h > 0} \| \varphi(h L) \|_{BMO_L(M) \to L^\infty(M)} & <&  \infty ,\\ 
\sup_{h > 0} \| \varphi(h L) \|_{L^1(M) \to H^1_L(M)} &<&  \infty,  \label{zeboite1}\\
\sup_{h>0} \|\varphi(hL) \|_{L^q(M)\to L^q(M)}&<&\infty,\qquad q\in[1,+\infty]. \label{bmobmo-q}
\end{eqnarray}
\end{theorem}
Let us first comment the assumption \eqref{bmocond}. We  note that \eqref{bmocond} for $k=1$ implies that $\varphi$ is uniformly continuous on $(0,+\infty)$ and hence admits a limit as $x$ tends to $0^+$.
For $k=0$, the condition $\int_{0}^{+\infty} \frac{|\varphi(x)|}{x}dx<+\infty$ forces the limit $\lim\limits_{x\rightarrow 0^+}\varphi(x)$ to be $0$. From the previous limit, one also deduces that $\varphi$ is bounded by $\int_{0}^{+\infty} |\varphi'(x)|dx$. The previous considerations ensure that the operator $\varphi(hL)$ is well-defined on $L^2(M)$ by the standard functional calculus. 

\begin{proof}[Proof of Theorem \ref{bmobmo}]
By duality, \eqref{bmobmo-a} cleary follows from \eqref{zeboite1}.
We will see  that \eqref{zeboite1} can be reduced to  $\varphi \in \CC_c^\infty(0, \infty)$. The general assumption \eqref{bmocond} will be reached by appealing to the weak factorization of  Corollary \ref{dixm3} (see Step 3 below). 
Similarly, we reduce \eqref{bmobmo-q}  to  $ \varphi \in \CC_c^\infty(0, \infty) $ in which case it follows from \eqref{mult}.  See  Step 4 below. 
  
 \medskip
\textbf{Step 1.} We start by proving \eqref{zeboite1} for $\varphi \in \CC_c^\infty(0, \infty)$. The first step consists by checking that \eqref{zeboite1} follows from the particular case $h=1$. \\
Following  \cite[pages 8-9]{HLMMY} the version of the Hardy space we need, denoted by $H^1_L(M)$, is defined by the square function $S_L$ as follows.  Set for every $x\in M$
\begin{equation}\label{e2.3}
S_{L}f(x):=\Big(\int_0^{\infty}\!\!\!\!\int_{\substack{ d(x,y)<t}}
| t^2L e^{-t^2L} f(y)|^2 {d\mu(y)\over V(x,t)}{dt\over
t}\Big)^{\frac{1}{2}}, 
\end{equation}
and denote the domain
$D :=\Big\{ f \in \overline{R(L)}: \ S_{L}f\in L^1(M) \Big\}$ where $\overline{R(L)}$ is the closure of the range of $L$ in $L^2(M)$. Then $H^1_{L}(M)$  is the completion of the space $D$ with respect to  the norm
\begin{equation}\label{defiHa}
 \|f\|_{H_{L}^1(M)}:=  \|S_L  f\|_{L^1(M) }.
\end{equation}
For any $h > 0$, we introduce the metric $d' := \frac{d}{\sqrt{h}}$ and the operator $L' := h L$. The measure $\mu$ is unchanged.  Then the volume with respect to $d'$ and $\mu$ is $V'(x,r) = V(x, r \sqrt{h})$. Obviously, the heat kernel of $L'$ satisfies the Gaussian bound $(G)$ with the same constants $C$ and $c$ but now with $V'(x, \sqrt{t})$ and $d'(x,y)$ instead of $V(x,\sqrt{t})$ and $d(x,y)$. We define $S_{L'}$, as in \eqref{e2.3}, by considering $d'$ and $V'$. Then a simple change of the variable shows the equalities
\begin{eqnarray*}
S_{L'} f(x)  &= & S_L f(x) \\
\| f \|_{H^1_{L'}(M)} &= &\| f \|_{H^1_L(M)}.
\end{eqnarray*}
By the same change of variable, we have $\| \varphi(h L) f \|_{H^1_{L}(M)} = \| \varphi(L') f \|_{H^1_{L'}(M)}$ for any $\varphi \in C_c^\infty(0, \infty)$. 
%This latter  equality is true for  $\varphi \in C_c^\infty([0, \infty))$ (we shall need this fact in Lemma \ref{sqrtL}).
%Note : la phrase précédente vient d'une ancienne version. C'est inutile.
  We have reduced the proof of  \eqref{zeboite1} to that of \begin{equation}\label{zeboite2}
\| \varphi(L') f \|_{H^1_{L'}(M)} \le C_0 \| f \|_{L^1(M) }
\end{equation} 
with some constant $C_0$ depending only on $\varphi$, the constants in the doubling property and the Gaussian bound $(G)$.

\medskip
\textbf{Step 2.} We prove \eqref{zeboite2} again for $\varphi \in \CC_c^\infty(0, \infty)$. We do this for $L$ instead of $L'$ for  simplicity of the notation. In order to prove this, we use the molecular decomposition (see \cite[page 7, Definition 2.3]{HLMMY}). Given a function $g \in L^2(M)$ and assume first that $g$ is supported in a ball
$B(x_B, r_B)$  with $r_B \ge 1$.  We prove that $\varphi(L) g$ is a multiple of a molecule of $H^1_L(M)$.   That is, there exist positive constants $m$ and $\varepsilon$, a ball $B = B(x'_B, r'_B)$ and a function $b$ such that $ \varphi(L) g$ is multiple of $L^m b$ and 
\begin{equation}\label{zeboite3}
 \| ({r'_B}^2 L)^k b \|_{L^2(U_j(B))} \le {r'_B}^{2m} 2^{-j\varepsilon} V(2^{j}B)^{-1/2}
 \end{equation}
 for $ k = 0,1,..., m$ and $j = 0,1,2,...$. Here 
$U_j(B) := 2^{j+1}B \setminus 2^j B$, $2^j B := B(x'_B, 2^j r'_B)$, $V(2^j B) :=  V(x'_B, 2^j r'_B)$ for $j \ge 1$ and $U_0(B) = B$.\\
We choose $B(x'_B, r'_B) = B(x_B, r_B)$ the same ball which contains the support of $g$ and take $b = L^{-m} \varphi(L) g$. The function $b$ exists since $L^{-m} \varphi(L)$ is a bounded operator on $L^2(M)$ because $\varphi$ is supported in $(0, \infty)$.  It remains to prove \eqref{zeboite3}. \\
We assume for simplicity that the support of $\varphi$ is contained in $[\frac{1}{2},1]$, the reasoning is the same for any $\varphi$ with compact support in $(0, \infty)$.
We have
\begin{eqnarray*}
\| ({r_B}^2 L)^k L^{-m} \varphi(L) g \|_{L^2(U_j(B))} &=& {r_B}^{2k} \| L^{k-m} \varphi(L) g \|_{L^2(U_j(B))}\\
& \le&  {r_B}^{2k} \| L^{k-m} \varphi(L)  \|_{L^1(B) \to L^2(U_j(B))} \| g \|_{L^1(M)}\\
&\le&  {r_B}^{2k} \| L^{k-m} \varphi(L)  \|_{L^2(U_j(B)) \to L^\infty(B)} \| g \|_{L^1(M)}\\
&\le&  {r_B}^{2m} \| L^{k-m} \varphi(L)  \|_{L^2(U_j(B)) \to L^\infty(B)} \| g \|_{L^1(M)}.
\end{eqnarray*}
Set $\psi(\lambda) := \lambda^{k-m} \varphi(\lambda)$ and fix $s > n/2$ where $n$ denotes again the homogeneous "dimension" in the doubling property \eqref{doublin}. We now apply Lemma 4.3 from \cite{DOS} (using \eqref{doublin} and $(G)$ and according to Lemma 2.2 from \cite{DOS}, the assumptions of  Lemma 4.3 hold for $p=\infty$). We then bound the last term as follows
\begin{eqnarray*}
&&{r_B}^{2m} \| L^{k-m} \varphi(L)  \|_{L^2(U_j(B)) \to L^\infty(B)} \| g \|_{L^1(M)} \\
 %& & \mbox{\textcolor{red}{est-ce que l'exposant de $(1+2^{j}r_B)^{-s}$ ne devrait pas etre $-s/2$ ?}}\\
& & \quad \le  
C_0 {r_B}^{2m} \sup_{y \in B} \frac{1}{ V(y,2)^{1/2}} (1 + 2^j r_B)^{-s} \| \psi \|_{W^{s  + 1, \infty}}  \| g \|_{L^1(M)}\\
& & \quad \le  C_0 {r_B}^{2m} V(2^j B)^{-1/2} \sup_{y \in B} \frac{ V(y, 2^j r_B)^{1/2}}{ V(y,2)^{1/2}} (1 + 2^j r_B)^{-s} \| \psi \|_{W^{s  + 1, \infty}}  \| g \|_{L^1(M)}\\
& & \quad \le C_0 C {r_B}^{2m} V(2^j B)^{-1/2} 2^{-j (s- n/2)}\| \psi \|_{W^{s + 1, \infty}}  \| g \|_{L^1(M)}.
\end{eqnarray*}
Let $C_\varphi := C_0C \| \psi \|_{W^{s  + 1, \infty}}$. The above inequality shows that $\frac{\varphi(L) g}{C_\varphi \|g\|_{L^1(M)}}$ is a molecule of $H^1_L(M)$. Hence
$$ \| \frac{\varphi(L) g}{C_\varphi \|g\|_{L^1(M)}} \|_{H^1_L(M)} \le 1.$$
Hence, for any $g\in L^2(M)$ with compact support, the following inequality holds
\begin{equation}\label{zeboite4}
\| \varphi(L) g \|_{H^1_L(M)} \le  C_\varphi \|g\|_{L^1(M)}.
\end{equation}
For a general $g\in L^1(M)$, classical density results ensure the existence of a sequence $(g_n)$ of $L^2(M)$, each $g_n$ has compact support, and $(g_n)$ converges to $g$ in $L^1(M)$. Then $\varphi(L)g_n \in H_L^1(M)$ for each $n$ and 
\begin{equation}\label{zeboite5}
\| \varphi(L) g_n \|_{H^1_L(M)} \le  C_\varphi \|g_n \|_{L^1(M)}.
\end{equation}
This implies that $(\varphi(L) g_n)$ is a Cauchy sequence in $H_L^1(M)$ and hence is convergent. On the other hand, since $\varphi(L)$ is bounded on $L^1(M)$ (see  \eqref{mult}), the limit of $(\varphi(L) g_n)$, in $L^1(M)$,  is $\varphi(L) g$. In addition, under \eqref{doublin} and $(G)$, the Hardy space $H_L^1(M)$ continuously embeds into $L^1(M)$ (see 
\cite[page 70]{HLMMY}). It follows that the  limit of  $(\varphi(L) g_n)$ in $H_L^1(M)$ is also $\varphi(L) g$. Taking the limit in \eqref{zeboite5} yields \eqref{zeboite4} for $g \in L^1(M)$. Thus, we have  proved \eqref{zeboite1} for $\varphi \in \CC_c^\infty(0, \infty)$.

\medskip
\textbf{Step 3.} Now we prove \eqref{zeboite1} for general functions $\varphi$ as in the theorem.
We apply   Corollary \ref{dixm3} to write (with the same notations) for any $x>0$
\begin{equation*}
\varphi(x) =\int_{0}^{+\infty} \Theta_1\left( \frac{x}{y}\right) F_1(y) dy+
\int_{0}^{+\infty} \Theta_2\left( \frac{x}{y}\right) F_2(y) dy.
\end{equation*}
By the functional calculus, we have for $h > 0$ 
\begin{equation*}
\varphi(hL)=
\int_{0}^{+\infty} \Theta_1\left(\frac{hL}{y} \right)F_1(y) dy+
\int_{0}^{+\infty} \Theta_2\left( \frac{hL}{y}\right) F_2(y)dy.
\end{equation*}
Hence, we can bound $\sup\limits_{h>0}\left\Vert \varphi(hL)\right\Vert_{L^1(M) \rightarrow H^1_L(M)}$ by 
\begin{equation*}
\sup\limits_{h>0}\int_{0}^{+\infty}
 \Big\Vert \Theta_1\Big(\frac{hL}{y} \Big) \Big\Vert_{L^1(M) \rightarrow H^1_L(M)} |F_1(y)|dy+ \Big\Vert\Theta_1\Big(\frac{hL}{y}\Big)  \Big\Vert_{L^1(M) \rightarrow H^1_L(M)}|F_2(y)| dy.
\end{equation*}
Remember now that we have proved \eqref{zeboite1} for functions in $\CC_c^\infty(0,\infty)$.
Since $\Theta_1$ and $\Theta_2$ are compactly supported in $(0,+\infty)$ and since $F_1$ and $F_2$ are integrable (thanks to Corollary \ref{dixm3}), the last upper bound is less than  
$C \left\Vert F_1\right\Vert_{L^1(0,+\infty)}+C\left\Vert F_2\right\Vert_{L^1(0,+\infty)}<+\infty$.

\medskip
\textbf{Step 4.} We explain here the proof of  \eqref{bmobmo-q}. For $\varphi$ being smooth with compact support, we have seen several times  that \eqref{bmobmo-q} holds (see \eqref{mult}).  We obtain   \eqref{bmobmo-q} under the condition \eqref{bmocond} by another application of Corollary \ref{dixm3}. 

The proof of Theorem \ref{bmobmo} is finished. 
\end{proof}
The multiplier theorem proved in this section will be used in a crucial way in the proof of the reverse Bernstein inequality. 

 \section{The reverse Bernstein inequality}\label{sec-reverse}
 In this section, we investigate reverse Bernstein inequalities and prove Theorem \ref{theo3-p}.
As in Section \ref{semi}, we introduce a semi-classical version of the reverse inequality. 

 \begin{theorem}\label{theo3-sc}
 Suppose \eqref{doublin} and $(G)$ and consider 
a function $\Psi \in \CC^\infty([0,+\infty))$ vanishing in a neighborhood of  $0$ and constant in a neighborhood of $+\infty$. Then we have the following assertions :
\begin{enumerate}[i) ]
\item For $q \in [2, \infty]$, the following semi-classical reverse Bernstein inequality holds
 \begin{equation*}
\frac{C}{ \sqrt{h}} \|  \Psi(hL) u \|_{q}  \leq  \| \nabla u \|_{q} +  \| \sqrt{W}  u  \|_{q},\qquad u\in L^2(M),\quad h>0.  \eqno(SRB_q)
  \end{equation*}
\item For $q\in (1,2)$, if $(R_p)$ holds for $p=\frac{q}{q-1}$ (or equivalently $(SB_p)$ holds, see Theorem \ref{theo2-psc}), then $(SRB_q)$ holds. 
\item For $q=1$, if $(R_\infty)$ and the condition $\inf\limits_{x\in M}V(x,1)>0$ hold, then $(SRB_1)$ holds. 
\end{enumerate}
\end{theorem}

%Point i) is clearly the most satisfactory statement since $(SRB_q)$ automatically yields. Actually, Point i) should be seen as the dual situation of Theorem \ref{theo1-psc}.
As in Section \ref{derivation}, the discrete reverse Bernstein inequality of Theorem \ref{theo3-p} is an almost straightforward consequence of the semi-classical version of  Theorem \ref{theo3-sc}.

\begin{proof}[Proof of  Theorem \ref{theo3-p}] 

\textit{i)} By considering a smooth function $\Psi$ satisfying
\begin{equation*}
\Psi(x) = \left\{ 
\begin{array}{ll}
0, \, \, x \le \frac{1}{2}\\
1,  \, \, x \ge 1.
\end{array}
\right.
\end{equation*}
we clearly have $ \Psi(hL) u = u$
for $ u = \sum\limits_{k=N}^{+\infty} \alpha_k \phi_k$  and $ h = \frac{1}{\lambda^2_N}$. We see that  $(RB_q)$ is a consequence of $(SRB_q)$.

For \textit{ii) and iii)} we note that we have proved the equivalences $(B_p)\Leftrightarrow (SB_p) \Leftrightarrow (R_{p})$ 
(see Section \ref{derivation} and Theorem \ref{theo2-psc}). So we   obtain as previously  assertions  ii) and iii) from  Theorem \ref{theo3-sc}.
\end{proof}

We now turn to the proof of Theorem \ref{theo3-sc}. The first thing to notice is the following consequence of Theorem \ref{bmobmo}. 

\begin{lemma}\label{sqrtL}
Assume \eqref{doublin} and $(G)$ and consider a smooth function $\Psi\in \CC^\infty([0,+\infty))$ vanishing near $0$ and being constant in a neighborhood of $+\infty$. Then for any $q\in [1,+\infty]$, 
\begin{eqnarray}\label{sqrtL1}
\left\Vert \Psi(hL) u \right\Vert_q & \leq & C \sqrt{h} \| \sqrt{L} \Psi(hL) u \|_{q} \\ \label{sqrtL2}
\left\Vert \Psi(hL) u \right\Vert_q & \leq & C' \sqrt{h} \| \sqrt{L} u\|_q.
\end{eqnarray}
For $q=+\infty$, the previous two bounds can be modified as follows
\begin{eqnarray}\label{sqrtL3}
\left\Vert \Psi(hL) u \right\Vert_\infty & \leq & C \sqrt{h} \| \sqrt{L} \Psi(hL) u \|_{BMO_L(M)} \\ \label{sqrtL4}
\left\Vert \Psi(hL) u \right\Vert_\infty & \leq & C' \sqrt{h} \| \sqrt{L} u\|_{BMO_L(M)}.
\end{eqnarray}
In the above estimates, $C$ and $C'$ may depend on $\Psi$ but are independent of $u\in L^2(M)$ and $h > 0$.
\end{lemma}
\begin{proof}
Note that  $\Psi(hL) $ is well defined on $L^2(M)$ by the functional calculus of $L$.
For simplicity we assume that $\Psi$ has support in $[1,+\infty)$. We  take $\varphi \in \CC^\infty(\R)$ such that
\[
\varphi(x) = \left\{ 
\begin{array}{ll}
0  \quad & \mbox{ for }x \le \frac{1}{2},\\
\frac{1}{\sqrt{x}} \quad & \mbox{ for }  x \ge 1.
\end{array}
\right.
\]
Hence we have $ \Psi(x) = \varphi(x) \sqrt{x} \Psi(x)$.
We apply Theorem \ref{bmobmo} to the function $\varphi$ and obtain \eqref{sqrtL1} for $q\in [1,+\infty]$ as follows 
\begin{eqnarray*}
 \| \Psi(hL) u \|_q &= & \| \varphi(hL) \sqrt{hL} \Psi(hL) u \|_q\\
 &\le & C\sqrt{h} \| \sqrt{L} \Psi(hL) u \|_{q}.
 \end{eqnarray*}
Similarly, we apply assertion  \eqref{bmobmo-a} of Theorem \ref{bmobmo} to obtain 
\begin{equation}\label{sqrtL5}
\left\Vert \Psi(hL) u \right\Vert_\infty  \leq  C \sqrt{h} \| \sqrt{L} \Psi(hL) u \|_{BMO_L(M)}.
\end{equation}
This proves \eqref{sqrtL4}. 
To see \eqref{sqrtL2} and \eqref{sqrtL4}, we merely notice that we can apply Theorem \ref{bmobmo} to the function $\Phi(x) :=\frac{1}{\sqrt{x}}\Psi(x)$ instead of $\varphi$. 
\end{proof}

%For the sake of clarity, we now explain how the following intuitive equivalence comes from the analysis of the paper \cite{russ2000}.
We recall the following standard lemma for which we give a proof for the sake of completeness. 
\begin{lemma}\label{russ}
Assume \eqref{doublin} and $(G)$ and fix $q\in (1,+\infty)$. Then for any $u\in \mathcal{D}(\sqrt{L})$ the following inequalities hold (where $p=\frac{q}{q-1}$)
\begin{equation}\label{dual-russ}
C\| \sqrt{L}  u \|_{q}\leq  \sup\limits_{\substack{ \| \sqrt{L} v \|_{p}\leq 1\\v\in \mathcal{D}(\sqrt{L})}} |( \sqrt{L}  u , \sqrt{L} v)| \leq  \| \sqrt{L}  u \|_{q}.
\end{equation}
For $q=1$ and $p=+\infty$, if one moreover assumes the condition $\inf\limits_{x\in M}V(x,1)>0$, then the previous inequalities are still true.
\end{lemma}
\begin{proof}
The upper bound in \eqref{dual-russ} follows from H\"{o}lder's inequality. We prove the converse. Note that 
\begin{equation}\label{dual1}\| \sqrt{L}  u \|_{q}\leq \sup\limits_{\substack{\left\Vert g\right\Vert_{p}\leq 1 \\g\in L^1(M)\cap L^\infty(M)}} |( \sqrt{L}  u , g )|.
\end{equation}
For $q\in (1,+\infty)$, it is well known that  
%one may invoke the appendix of \cite[Lemmas 8 and 9]{russ2000} so that the following limit yields
\begin{equation}\label{Lp-lim}
\lim\limits_{t\rightarrow +\infty} \left\Vert e^{-tL}g - \Pi_0(g)\right\Vert_{L^p(M)}=0, \qquad \forall g\in L^1(M)\cap L^\infty(M),
\end{equation}
where $\Pi_0:L^2(M)\rightarrow L^2(M)$ is the orthogonal projection on $\ker(L)$ (see, e.g., \cite[Lemmas 8 and 9]{russ2000}). 
Also $\Pi_0$ maps $L^1(M)\cap L^\infty(M)$ into itself. %thanks to \cite[Lemma 9]{russ2000}.
For $q=1$,  \eqref{Lp-lim} still holds true under the additional condition $\inf\limits_{x\in M} V(x,1)>0$. Actually, the Gaussian estimates $(G)$ imply the finiteness of  $\sup\limits_{x\in M}|p_{1}(x,x)|$. Consequently, $e^{-L}$ is hypercontractive in the sense that it sends $L^1(M)$ into $L^\infty(M)$. By interpolation, $e^{-L}$ also sends $L^2(M)$ into $L^\infty(M)$ so that \eqref{Lp-lim} gives 
\begin{equation*}
 \lim\limits_{t\rightarrow +\infty} \left\Vert e^{-tL}g -e^{-L} \Pi_0(g)\right\Vert_{L^\infty(M)}=0, \qquad \forall g\in L^1(M)\cap L^\infty(M).
\end{equation*}
Since  $e^{-L}\Pi_0(g)= \Pi_0(g)$, we see  that  \eqref{Lp-lim} holds for $p=\infty$ as well.

Since $\ker(L)=\ker(\sqrt{L})$ we have from \eqref{dual1}
\begin{equation}\label{dual2}\| \sqrt{L}  u \|_{q}\leq \sup\limits_{\substack{\left\Vert g\right\Vert_{p}\leq 1 \\g\in L^1(M)\cap L^\infty(M)}} |( \sqrt{L} u , g- \Pi_0(g) )|, 
\end{equation}
and hence
%Moreover, \eqref{Lp-lim} ensures the limit $\lim\limits_{t\rightarrow +\infty}g-e^{-tL}g=g-I(g)$ in $L^p(M)$. Such considerations lead to the following bound
\begin{equation}\label{dual3}
\| \sqrt{L} u \|_{q}\leq  \sup\limits_{\substack{\left\Vert g\right\Vert_{p}\leq 1 \\g\in L^1(M)\cap L^\infty(M)}}\sup\limits_{t>0} |( \sqrt{L}u,g-e^{-tL}g) |. 
\end{equation}
%But, if one considers $g$ as an element of $L^2(M)$, we see that 
Since $g-e^{-tL}g= -L\displaystyle\int_{0}^t e^{-sL}g ds$
belongs to the range of $L$ (and thus to the range of $\sqrt{L}$), we obtain from the contractivity of the semigroup $(e^{-tL})_{t\geq 0}$ on $L^p(M)$ that  
\begin{equation*}
\left\Vert g-e^{-tL} g\right\Vert_{p}\leq  2 \left\Vert g\right\Vert_{p},\quad  t>0.
\end{equation*}
This gives  \eqref{dual-russ}.
\end{proof}

The previous preliminary results allow us to prove Points ii) and iii) of Theorem \ref{theo3-sc}. Actually, our next proof  shows that  the implication $(SB_p) \Rightarrow (SRB_q)$ holds for all $1\leq q<+\infty$ (with the additional assumption $\inf\limits_{x\in M}V(x,1)>0$ for $q=1$).

\begin{proof}[Proof of Theorem \ref{theo3-sc}]
We start with the proof of assertions  ii) and iii).  Thanks to Theorem \ref{rev-univ} below, me merely have to prove $(SRB_q)$ for $\Psi$ having compact support in $(0,+\infty)$.
We now assume $(R_p)$ for $p=\frac{q}{q-1}$, or equivalently $(SB_p)$ (thanks to Theorem \ref{theo2-psc}) and we prove $(SRB_q)$. By Lemma \ref{sqrtL}, it suffices to prove 
\begin{equation}\label{SRBq-bakry}
C \|\sqrt{L} \Psi(hL) u \|_{q}\leq  \|\nabla u \|_q+\|\sqrt{W}u \|_q, \qquad u\in L^2(M),\quad h>0.
\end{equation}
Lemma \ref{russ} implies 
\begin{eqnarray*}
\|\sqrt{L} \Psi(hL) u \|_{q} & \leq & C \sup\limits_{\substack{ \| \sqrt{L} v \|_{p}\leq 1\\v\in \mathcal{D}(\sqrt{L})}} | ( \sqrt{L}  \Psi(hL)u , \sqrt{L} v )| \\
& \leq &  C \sup\limits_{\substack{ \| \sqrt{L} v \|_{p}\leq 1\\v\in \mathcal{D}(\sqrt{L})}} | ( \sqrt{L}  u , \sqrt{L} \Psi(hL)v )| \\
& \leq & C  \sup\limits_{\substack{ \| \sqrt{L} v \|_{p}\leq 1\\v\in \mathcal{D}(\sqrt{L})}}  \left\vert\int_{M} \nabla u\cdot \nabla \{ \Psi(hL)v \}+ W u \{\Psi(hL)v\} dx \right\vert\\
& \leq & C \left( \| \nabla u\|_{q}+  \|\sqrt{W} u\|_{q}\right)\sup\limits_{\substack{ \| \sqrt{L} v \|_{p}\leq 1\\v\in \mathcal{D}(\sqrt{L})}}\| \nabla \{ \Psi(hL)v \}\|_p+\| \sqrt{W}\{\Psi(hL)v\}\|_p.
\end{eqnarray*}
%The Bakry duality argument reads as follows
%\begin{eqnarray*}
%\|\sqrt{L} \Psi(hL) u \|_{q} & \lesssim &\sup\limits_{\substack{ \| \sqrt{L} v \|_{p}\leq 1\\v\in \mathcal{D}(\sqrt{L})}}  \| \nabla u\|_{q} \| \nabla \{ \Psi(hL)v \}\|_p+ \|%\sqrt{W} u\|_{q}\| \sqrt{W}\{\Psi(hL)v\}\|_p \\
%& \lesssim & \left( \| \nabla u\|_{q}+  \|\sqrt{W} u\|_{q}\right)\sup\limits_{\substack{ \| \sqrt{L} v \|_{p}\leq 1\\v\in \mathcal{D}(\sqrt{L})}}\| \nabla \{ \Psi(hL)v \}\|_p+\| %\sqrt{W}\{\Psi(hL)v\}\|_p.
%\end{eqnarray*}
We now consider $\psi:\R\rightarrow \R$ a smooth function with support in $(0,+\infty)$ which equals $1$ on the support of $\Psi$.
Since $\Psi(hL) =\psi(hL)\Psi(hL)$  we may apply the semi-classical Bernstein inequality $(SB_p)$ to obtain
\begin{equation*}
\|\sqrt{L} \Psi(hL) u \|_{q}  \leq C \left(\| \nabla u\|_{q}+  \|\sqrt{W} u\|_{q}\right)\sup\limits_{\substack{ \| \sqrt{L} v \|_{p}\leq 1\\v\in \mathcal{D}(\sqrt{L})}} \frac{1}{\sqrt{h}}  \| \Psi(hL)v \|_p.
\end{equation*}
To finish the proof of \eqref{SRBq-bakry}, we use  \eqref{sqrtL2} from  Lemma \ref{sqrtL}.

It remains to prove assertion i). 
As noticed above, the previous proof is also valid  for $q\in [2,\infty)$ (and hence $1<p\leq2$) and shows the implication $(SB_p)\Rightarrow (SRB_q)$.
Since by Theorem \ref{theo1-psc} the semi-classical Bernstein inequality $(SB_p)$ (or equivalently $(R_p)$) holds, we see that assertion i) of Theorem \ref{theo3-sc} holds for $q\in [2,+\infty)$.
The case $q=+\infty$ (for which $p=1$) seems to be more complicate. The previous proof breaks down\footnote{even in the Euclidean case, $e^{t\Delta}g$ does not converge in $L^1(\R^d)$ as $t\rightarrow +\infty$ for positive $g\in L^1(\R)$.} at  \eqref{Lp-lim} for $p=1$. We shall postpone the main argument to the next proposition in which we prove 
\begin{equation}\label{truc2}
C\, \| \sqrt{L} u \|_{BMO_L(M)}\leq \| \nabla u \|_\infty + \|  \sqrt{W}  u \|_\infty.
\end{equation}  
Now, by  \eqref{sqrtL4} of Lemma \ref{sqrtL}, we see that assertion i) with $q= \infty$ follows from \eqref{truc2}.
\end{proof}

%%%%%%%%%%%%%%%
%We shall overcome such an issue with the following two ideas :
%\begin{enumerate}[$\bullet$]
%\item due to the inequality \eqref{sqrtL4} of Lemma \ref{sqrtL}, the endpoint case $q=+\infty$ of Theorem \ref{theo3-sc} will be a consequence of the following inequalities
%\begin{equation*}
%C\left\Vert \sqrt{L} u \right\Vert_{BMO_L(M)}\leq \left\Vert \nabla u \right\Vert_{\infty}+\| \sqrt{W} u\|_{\infty}.
%\end{equation*}
%\item the above inequalities are proved below by exploiting the $H_L^1(M)\rightarrow L^1(M)$ boundedness of the 
%Riesz transform type operators $\nabla L^{-1/2}$ and $\sqrt{W} L^{-1/2}$ (proved in \cite{DOY2006}) and a version of the Bakry duality argument.
%\end{enumerate}

%The previous explanations and the following proposition achieve the proof of Theorem \ref{theo3-sc}.
\begin{proposition}\label{proH}
Suppose \eqref{doublin} and $(G)$. Denote by $\mathcal{D}_\infty(M)$ the subspace of distributions  $u$ on $M$ satisfying $\| \nabla u \|_\infty + \|  \sqrt{W}  u \|_\infty<+\infty$.
The operator $\sqrt{L}:\mathcal{D}(\sqrt{L})\cap \mathcal{D}_\infty(M)\rightarrow L^2(M)$ can be extended to 
an operator $\sqrt{L}:\mathcal{D}_\infty(M)\rightarrow BMO_L(M)$ which satisfies \eqref{truc2} for $u\in\mathcal{D}_\infty(M)$.
\end{proposition}

We  recall some facts  about finite molecular decomposition in Hardy spaces (already used in the proof of Theorem \ref{bmobmo}).
The notion of $(1,2,m,\varepsilon)$-molecules is defined in \cite[Definition 2.3]{HLMMY} where $m$ is an integer satisfying $m>\frac{n}{4}$ ($n$ being is the exponent in \eqref{doublin}) and $\varepsilon>0$ is arbitrary.
We  denote by $H_{L,mol}^{1,f}(M)$ the space of finite linear combinations of $(1,2,m,\varepsilon)$-molecules.
We forget $m$ and $\varepsilon>0$ in our notations for simplicity. Then we have
\begin{lemma}\label{resumH1}
The following properties hold
\begin{enumerate}[a) ]
\item The subspace $H_{L,mol}^{1,f}(M)$ is dense in the Hardy space $H_L^1(M)$.
%\item For any $\varphi\in \CC_c^\infty(0,+\infty)$ and any $g\in L^2(M)$ with support in a ball $B$ of radius $r_B\leq 1$,  $\varphi(L)g$ belongs to $H_{L,mol}^{1,f}(M)$.
\item The subspace $H_{L,mol}^{1,f}(M)$ is contained in $L^2(M)$. Moreover, $H_{L,mol}^{1,f}(M)$ is contained in the range $R(\sqrt{L})$ of $\sqrt{L}$, i.e.,  any $w \in H_{L,mol}^{1,f}(M)$ can be written as 
\begin{equation}\label{facto}
w=\sqrt{L}v \quad \mbox{with} \quad v\in \mathcal{D}(\sqrt{L}),
\end{equation}
\item[ ] Finally, the subspace $H_{L,mol}^{1,f}(M)$ is dense in $R(\sqrt{L})$ for the $L^2(M)$-norm.
\end{enumerate}
\end{lemma}
\begin{proof}
a) See  \cite[Definition 2.4 and Chapter 7]{HLMMY}. 

%b) This is proved in  Step 2 in the proof of Theorem \ref{bmobmo}.

b) From \cite[Definition 2.3, Point i)]{HLMMY}, $w$ belongs to the domain of $L^m$. Hence $w$ can be written $w=\sqrt{L} L^{m-\frac{1}{2}}w'$ with $L^{m-\frac{1}{2}}w'\in L^2(M)$. In particular, the inclusions $H_{L,mol}^{1,f}(M) \subset R(\sqrt{L})\subset L^2(M)$ hold. For the last assertion of the statement, we have not found a reference for it, so we give a proof.

Let us consider $v\in \mathcal{D}(\sqrt{L})$ and we want to prove that $\sqrt{L}v$ is a limit, in $L^2(M)$, of a sequence of elements in $H_{L,mol}^{1,f}(M)$.
To do so, we consider a sequence of functions $\varphi_k\in \CC_c^\infty(0,+\infty)$ satisfying 
\begin{equation}\label{limvarp}
0\leq \varphi_k \leq 1 \qquad  \mbox{and}\qquad \lim\limits_{k\rightarrow +\infty} \varphi_k(t)=\chi_{(0,+\infty)}(t), \qquad t\in [0,+\infty).
\end{equation}
Since $\sqrt{L} v$ is orthogonal to $\ker(\sqrt{L})$ we have $\sqrt{L}v=\chi_{(0,+\infty)}(\sqrt{L})\sqrt{L}v$. Then the spectral theorem for the unbounded self-adjoint operator $\sqrt{L}$ (see for instance \cite[Theorems 13.24 and 13.30]{rudin-fun}) ensure that there is a finite measure $\nu$ on $[0,+\infty)$, indeed the spectral measure 
$dE_{\sqrt{L}v,\sqrt{L}v}$ at the couple $(\sqrt{L}v,\sqrt{L}v)$, such that 
\begin{eqnarray*}
\left\Vert \sqrt{L} v- \varphi_k(\sqrt{L}) \sqrt{L}v \right\Vert_{L^2(M)}^2 & =& \left\Vert \Big(\chi_{(0,+\infty)}(\sqrt{L})- \varphi_k(\sqrt{L}) \Big) \sqrt{L}v \right\Vert_{L^2(M)}^2   \\
& =& \int_{[0,+\infty)} |\chi_{(0,+\infty)}(t)-\varphi_k(t)|^2 d\nu(t).
\end{eqnarray*}
The dominated convergence theorem and the conditions \eqref{limvarp} ensure that the previous term tends to $0$.
As a consequence of  $\sqrt{L}\varphi_k(\sqrt{L})=\varphi_k(\sqrt{L})\sqrt{L}$, we obtain  the following limit in $L^2(M)$
\begin{equation*}
\sqrt{L}v=\lim\limits_{k\rightarrow +\infty}  \sqrt{L}\varphi_k(\sqrt{L})v. 
\end{equation*}
Set $\psi_k(x) := \sqrt{x} \varphi_k(\sqrt{x})$ for $x \ge 0$ and approximate $v$, in $L^2(M)$, by a sequence $(v_n)_n \in L^2(M)$ and each $v_n$ has compact support. Note that  $\psi_k(L) v_n \in H_{L,mol}^{1,f}(M)$ as shown in Step 2 in the proof of Theorem \ref{bmobmo}. Since 
$\sqrt{L} v$  is the limit in $L^2(M)$ (as $n, k \to \infty$) of $\psi_k(L) v_n$ we conclude that $\sqrt{L} v$ is in the closure of $H_{L,mol}^{1,f}(M)$ in $L^2(M)$.
\end{proof}

\begin{proof}[Proof of Proposition \ref{proH}]
As mentioned previously, the Riesz transforms  $\nabla L^{-1/2}$ and $\sqrt{W}L^{-1/2}$ are bounded from $H^1_L(M)$ into $L^1(M)$ (see \cite{DOY2006} in which $M = \R^n$ but the arguments merely need the assumptions \eqref{doublin} and $(G)$). We also stress that the definition of the Hardy space used in \cite{DOY2006} coincides with \eqref{defiHa}, namely that of \cite{HLMMY}. Therefore, 
\begin{equation}\label{ineq-1}
\| \nabla u \|_1 + \|  \sqrt{W}  u \|_1  \le c\, \| \sqrt{L} u \|_{H^1_L(M)}.
\end{equation}
%Once the realization of the Laplacian is given, we know that the domain of the closed quadratic form of $L=-\Delta+W$
%turns out to be the domain $\mathcal{D}(\sqrt{L})$ of $\sqrt{L}$ (where $\sqrt{L}$ is defined by functional calculus, see \cite[Theorem 8.1]{maati}).
Moreover, the following formula holds
\begin{equation}\label{formq}
( \sqrt{L} u,\sqrt{L}v ) = \int_{M} \nabla u \cdot \nabla v +Wuv d\mu ,\qquad (u,v)\in \mathcal{D}(\sqrt{L})\times \mathcal{D}(\sqrt{L}).
\end{equation}
We now want to reach the following two goals 
\begin{itemize}
\item[\textbf{A}) ] give a reasonable definition of $\sqrt{L}$ as an operator from $\mathcal{D}_\infty(M)$ into $BMO_L(M)$, where $\mathcal{D}_\infty(M)$ is the subspace of distributions $u$ on $M$ satisfying
\begin{equation}\label{defi-Dinf}
|\nabla u| \in L^\infty(M)\quad \mbox{and}\quad \sqrt{W} u\in L^\infty(M).
\end{equation}
\item[\textbf{B}) ] prove that the previous reasonable definition is compatible with the initial definition $\sqrt{L}$ as an operator on $L^2(M)$. In other words, the two definitions must coincide on $\mathcal{D}(\sqrt{L})\cap \mathcal{D}_\infty(M)$.
\end{itemize}
We use the properties in Lemma \ref{resumH1} of $H_{L,mol}^{1,f}(M)$. 
Due to the inclusion $H_{L,mol}^{1,f}(M) \subset L^2(M)$, we can consider the following subspace of  $\mathcal{D}(\sqrt{L})$
\begin{equation*}
\mathcal{D}_{mol}(\sqrt{L}):=\{v\in \mathcal{D}(\sqrt{L}),\quad \sqrt{L}v \in H_{L,mol}^{1,f}(M)\}.
\end{equation*}
%Actually such a definition allows to overcome the issue of studying the definition of $\sqrt{L}$ as an operator with domain in the Hardy space $H_{L}^1(M)$.
Let us now fix $u\in \mathcal{D}_\infty(M)$ and consider the linear functional $T_u:  \mathcal{D}_{mol}(\sqrt{L}) \rightarrow  \R$ defined by
\begin{equation*}
( T_u,v ) :=\int_{M} \nabla u \cdot \nabla v +Wuv d\mu, \qquad v\in \mathcal{D}_{mol}(\sqrt{L}).
\end{equation*}
By \eqref{ineq-1} and \eqref{defi-Dinf} it is clear that $( T_u,v )$ is an absolutely convergent integral and satisfies 
\begin{equation}\label{tokeep}
|( T_u,v) |\leq C  (\left\Vert \nabla u \right\Vert_{L^\infty(M)}+\Vert \sqrt{W}u\Vert_{L^\infty(M)})\Vert\sqrt{L} v\Vert_{H_L^1(M)}. 
\end{equation}
Let us now explain why such an estimate ensures the existence of a unique element $b_u\in BMO_L(M)$ satisfying
\begin{equation}\label{bmo-lin1}
( T_u,v ) = \langle b_u,\sqrt{L} v \rangle_{BMO-H^1},
\end{equation} 
where the bracket in the RHS is the duality between $BMO_L(M)$ and $H_L^1(M)$.
Moreover, the element $b_u\in BMO_L(M)$ will satisfy
\begin{equation}\label{bmo-lin2}
\left\Vert b_u\right\Vert_{BMO_L(M)}\leq C \big(\left\Vert \nabla u \right\Vert_{L^\infty(M)}+\Vert \sqrt{W}u\Vert_{L^\infty(M)}\big).
\end{equation}
Once we establish that $b_u = \sqrt{L} u$ we obtain \eqref{truc2} from \eqref{bmo-lin2}. \\
For  $w\in H_{L,mol}^{1,f} (M)$ we introduce the following affine subspace of $\mathcal{D}_{mol}(\sqrt{L})$
\begin{equation}\label{defiA}
\mathbb{A}(w):=\{v\in \mathcal{D}(\sqrt{L}),\quad \sqrt{L}v=w\}.
\end{equation}
We note that $\mathbb{A}(w)$ is not empty thanks to \eqref{facto}.
Moreover, it is clear that \eqref{tokeep} shows the following equality
\begin{equation*}
( T_u, v ) -( T_u,v')= ( T_u,v-v')=0, \qquad (v,v')\in \mathbb{A}(w)\times \mathbb{A}(w).
\end{equation*}
In other words, $( T_u, v)$ is independent of the choice of $v\in \mathbb{A}(w)$ and merely depends on $w$. Furthermore, it is clear that the following application 
\begin{equation}
w\in H_{L,mol}^{1,f}(M)\mapsto ( T_u, v)
\end{equation}
is linear and satisfies, thanks to \eqref{tokeep} and \eqref{defiA}, the upper bound
\begin{equation*}
|( T_u, v)|\leq C  (\left\Vert \nabla u \right\Vert_{L^\infty(M)}+\Vert \sqrt{W}u\Vert_{L^\infty(M)})\Vert w \Vert_{H_L^1(M)}.
\end{equation*}
By using the density of $H_{L,mol}^{1,f}(M)$ in $H_{L}^1(M)$ (see Lemma \ref{resumH1}) and the important fact that $BMO_L(M)$ is the dual space of $H_{L}^1(M)$ (see \cite[Theorem 6.4]{HLMMY}), we see  that there exists a unique element $b_u\in BMO_L(M)$ satisfying 
\begin{equation*}
(T_u,v)=\langle b_u, w\rangle_{BMO-H^1}=\langle b_u,\sqrt{L} v \rangle_{BMO-H^1},\end{equation*}
 that is \eqref{bmo-lin1} and \eqref{bmo-lin2}.
The formulas \eqref{formq} and \eqref{bmo-lin1} clearly suggest to define $b_u:=\sqrt{L}u$ so that $\sqrt{L}$ would be defined from $\mathcal{D}_\infty(M)$ into $BMO_L(M)$. That is our gaol \textbf{A} above.

For consistency with the definition of $\sqrt{L}$ on $L^2(M)$, we also have to reach the second  goal \textbf{B}. 
We now assume $u$ to belong in $\mathcal{D}(\sqrt{L})\cap \mathcal{D}_\infty(M)$. The compatibility of the previous two definitions will come from Lemma \ref{resumH1}. For any $w\in H_{L,mol}^{1,f}(M)$, the previous construction allows us to write $w=\sqrt{L} v$ for some  $v\in \mathcal{D}(\sqrt{L})$ and 
\begin{equation*}
\langle b_u,w \rangle_{BMO-H^1}= \langle b_u,\sqrt{L}v \rangle_{BMO-H^1} = \int_{M} \nabla u\cdot \nabla v+W uv d\mu.
\end{equation*}
But \eqref{formq} proves the equality $\langle b_u,w \rangle_{BMO-H^1}=( \sqrt{L} u,w)_{L^2-L^2}$.
The density of $H_{L,mol}^{1,f} (M)$ in $H_L^1(M)$, for the $H_L^1(M)$-norm and in the range $R(\sqrt{L})$ of $\sqrt{L}$, for the $L^2(M)$-norm, ensures the compatibility of $b_u$ and $\sqrt{L}u$.
Thus, the equality  $b_u=\sqrt{L}u$ is true for $u$ satisfying \eqref{defi-Dinf}.
\end{proof}

Finally, as in  Theorem \ref{theo2-psc} we prove that the reverse semi-classical Bernstein inequality $(SRB_q)$ is independent of the choice of the function $\Psi$. More precisely, 

\begin{theorem}\label{rev-univ}
Suppose \eqref{doublin} and $(G)$, then for any $q\in [1,+\infty]$ the following assertions are equivalent
\begin{enumerate}[i) ]
\item there exists a non-zero function $\Psi_0\in\CC^\infty([0,+\infty))$ vanishing near $0$ and being constant near $+\infty$ such that $(SRB_q)$ holds,
\item for any $\Psi \in \CC^\infty([0,+\infty))$ vanishing near $0$ and being constant near $+\infty$, $(SRB_q)$ holds,
\item for any $\beta>\frac{1}{2}$, the inequality $(SRB_q)$ holds for $\Psi:x\mapsto x^{\beta} e^{-x}$,
\item there exists $\beta>\frac{1}{2}$ such that $(SRB_q)$ holds for $\Psi:x\mapsto x^{\beta} e^{-x}$.
\end{enumerate}
\end{theorem}
\begin{proof}
i) $\Rightarrow$ ii) The idea is completely similar to that used in Section \ref{sec-Riesz}. 
One checks that i) and the multiplier estimates \eqref{mult} imply that $\Psi_0^2$ also satisfies $(SRB_q)$, that is,
\begin{equation*}
\frac{1}{\sqrt{h}}\| \Psi_0 (hL)\Psi_0(hL)u\|_q \leq C \frac{1}{\sqrt{h}}\| \Psi_0(hL)u\|_q \leq C' \| \nabla u\|_q+\|\sqrt{W} u\|_q.
\end{equation*}
As in Step 2 of Section \ref{sec-Riesz}, we similarly check that any $\Psi \in \CC_c^\infty(0,+\infty)$ also satisfies $(SRB_q)$.
We finish as in Step 3 of Section \ref{sec-Riesz} the proof of  ii) by applying Corollary \ref{dixm3}. Indeed, 
the smooth function $F:(0,+\infty)\rightarrow \R$ defined by $F(x)=\frac{1}{\sqrt{x}} \Psi(x)$ clearly satisfies the condition \eqref{dixm3-h}.
Hence 
\begin{equation*}
\frac{1}{\sqrt{x}}\Psi(x)=\int_{0}^{+\infty} \Theta_1\left(\frac{x}{y}\right)F_1(y) dy+
\int_{0}^{+\infty} \Theta_2\left(\frac{x}{y}\right)F_2(y)dy,
\end{equation*}
with  some $\Theta_1,\Theta_2\in \CC_c^\infty(0,+\infty)$ and $F_1,F_2\in L^1(0,+\infty)$.
Similarly to Section \ref{sec-Riesz}, we modify $\Theta_1$ and $\Theta_2$ as follows 
\begin{equation*}
\Psi_1(x)=\sqrt{x}\Theta_1(x) \qquad \mbox{and} \qquad \Psi_2(x)=\sqrt{x}\Theta_2(x).
\end{equation*}
We note that $\Psi_1$ and $\Psi_2$ are smooth and compactly supported in $(0,+\infty)$ and we obtain for $h>0$
\begin{eqnarray*}
\Psi(x)&= & \int_{0}^{+\infty}  \sqrt{y}\Psi_1\left( \frac{x}{y}\right)F_1(y)dy+\int_{0}^{+\infty} \sqrt{y}\Psi_2\left( \frac{x}{y}\right) F_2(y)dy\\
\Psi(hL) & =& \int_{0}^{+\infty}  \sqrt{y}\Psi_1\left( \frac{h}{y}L\right)F_1(y)dy+\int_{0}^{+\infty} \sqrt{y}\Psi_2\left( \frac{h}{y}L\right) F_2(y)dy.
\end{eqnarray*}
The validity of the semi-classical reverse Bernstein inequality $(SRB_q)$ for $\Psi$ is now a consequence of the beginning of the proof ensuring that $(SRB_q)$ holds for the compactly supported functions $\Psi_1$ and $\Psi_2$. More precisely, 
\begin{eqnarray*}
\left\Vert \Psi(hL) u \right\Vert_{q} & \leq &
\int_{0}^{+\infty}  \sqrt{y} \Big\| \Psi_1\left( \frac{h}{y}L\right) u \Big\|_{q} |F_1(y)|dy+\int_{0}^{+\infty} \sqrt{y}\big\|\Psi_2\left( \frac{h}{y}L\right) u\big\|_{q} |F_2(y)|dy \\
& \leq & C \sqrt{h} \big(  \|\nabla u \|_{q}+\|\sqrt{W} u \|_q\big)    \int_{0}^{+\infty} |F_1(y)|+|F_2(y)|dy \\
& \leq & C \sqrt{h} \big(  \|\nabla u \|_{q}+\|\sqrt{W} u \|_q\big).
\end{eqnarray*}

ii) $\Rightarrow$ i) is obvious.

i) $\Rightarrow$ iii) We take  $F(x)=\frac{1}{\sqrt{x}}x^\beta e^{-x}$ and argue as above.

iii) $\Rightarrow$ iv) is obvious.

iv) $\Rightarrow$ i) We choose $\Psi_0$ with compact support in $(0,+\infty)$ and we factorize $\Psi_0(x)=x^\beta e^{-x} \Psi(x)$ with $\Psi\in \CC_c^\infty(0,+\infty)$. Hence, we may use \eqref{mult} to obtain
\begin{equation*}
\frac{1}{\sqrt{h}}\| \Psi_0(hL)u\|_q \leq C \frac{1}{\sqrt{h}} \|  (hL)^\beta e^{-hL} u\|_q\leq C' \|\nabla u\|_q+\|\sqrt{W}u\|_q.
\end{equation*}

\end{proof}

 \section{Examples}\label{sec-examples}
Our results apply to a wide class of differential operators. We shall focus on Schr\"odinger operators, elliptic operators on domains and compact manifolds. 

\medskip
\noindent {1. \it Schr\"odinger operators.}  Let $M = \R^n$ and $0 \le W \in L_{loc}^1$. Then obviously, the heat kernel of $L = -\Delta + W$ satisfies the Gaussian upper bound
$$ 0 \le p_t(x,y) \le \frac{1}{(4 \pi t)^{n/2}} e^{- \frac{|x-y|^2}{4t}}.$$
Thus, Theorem \ref{theo1-psc} applies to $L$. If in addition, the spectrum of $L$ is discrete then we obtain the Bernstein inequality $(B_p)$ on $L^p(\R^n)$ for $p \in [1, 2]$.\\
Let now $q > 1$ and recall that $W$ belongs to the {\it reverse H\"older class} 
$RH_q$ if there exists a constant $C > 0$ such that 
$$ \left( \frac{1}{| Q|} \int_Q W^q(x) dx \right)^{1/q}  \le \frac{C}{|Q|} \int_Q W(x) dx$$
for every cube $Q$. In this case, the Riesz transforms $\nabla(-\Delta + W)^{-1/2}$ and $\sqrt{W} (-\Delta + W)^{-1/2}$ are  bounded on $L^p(\R^n)$ for $ p \in (2, 2q + \varepsilon)$ for some $\varepsilon > 0$. See \cite{AB07} and \cite{Sh95}. Thus the regularity property $(R_p)$ is satisfied and we can  
apply Theorem \ref{theo2-psc} to obtain the semi-classical Bernstein inequality
$$ \| \nabla \psi(hL) \|_{p \to p} +  \| \sqrt{W} \psi(hL) \|_{p \to p} \le \frac{C}{\sqrt{h}}$$
for $\psi \in \CC_c^\infty([0, \infty))$ and $p \in (2, 2q + \varepsilon)$.  Again if $L$ has discrete spectrum then we have the discrete  Bernstein inequality.

If $M$ is a general non-compact Riemannian manifold such that \eqref{doublin} and $(G)$ are satisfied then one can find in \cite{Assad-Ouhabaz2012} conditions on $W$ which imply the boundedness on $L^p(M)$ for  some $p > 2$ of the Riesz transforms $\nabla(-\Delta + W)^{-1/2}$ and $\sqrt{W} (-\Delta + W)^{-1/2}$. One of the  conditions there is an integrability condition of the type
$$ \int_0^1 \left \Vert \frac{\sqrt{W}}{V(., \sqrt{t})^{1/r}} \right \Vert_{r} \frac{dt}{\sqrt{t}} + \int_1^\infty  \left \Vert \frac{\sqrt{W}}{V(., \sqrt{t})^{1/s}} \right \Vert_{s} \frac{dt}{\sqrt{t}} < \infty$$
for some values $r, s > 2$ which depend on $p$. We refer to \cite{Assad-Ouhabaz2012} for the precise statements.

\medskip
\noindent {2. \it The harmonic oscillator.}  Let $M= \R^n$ and $L := -\Delta + |x|^2$ be the harmonic oscillator. Since $W(x) = |x|^2$ is non-negative, we have immediately from Theorem \ref{theo1-p} the Bernstein inequality $(B_p)$ for all $p \in [1,2]$ as well as the reverse Bernstein inequality $(RB_q)$ for $q \in [2, \infty]$ by Theorem \ref{theo3-p}.
In order to reach the cases $p\in (2,+\infty]$ and $q\in [1,2)$ we have to prove the regularity property $(R_\infty)$ (note that the assumption $\inf\limits_{x\in \R^n} V(x,1)>0$ of Theorem \ref{theo3-p} is obvious).
Actually, $(R_\infty)$ is clearly equivalent to the following estimate
\begin{equation}\label{Rinfini}
\sum_{k=1}^n \| \partial_{x_k} e^{-tL}\|_{\infty\to \infty}+\|x_k e^{-tL}\|_{\infty\to \infty}\leq \frac{C}{\sqrt{t}},\qquad t>0.
\end{equation}
Let $p_{t}(x,y)$ be the heat kernel of $-\Delta+|x|^2$. Note that  
\begin{equation}\label{ppr}
p_{t}(x,y)=\prod_{k=1}^n \wp_{t}(x_k,y_k),
\end{equation}
where $\wp_{t}$ is heat kernel in dimension $1$. By Mehler's  formula
\begin{equation*}
\wp_t(x_k,y_k)=\frac{1}{\sqrt{2\pi\sinh(2t)}}\exp\left(-\frac{\tanh(t)}{4}(x_k+y_k)^2-\frac{(x_k-y_k)^2}{4\tanh(t)} \right),
\end{equation*}
which directly extends to the multidimensional case
\begin{equation}\label{Mahler}
p_t(x,y)=\frac{1}{(2\pi\sinh(2t))^{n/2}}\exp\left(-\frac{\tanh(t)}{4}|x+y|^2-\frac{|x-y|^2}{4\tanh(t)} \right).
\end{equation}
For any $t>0$, we obtain  $p_{t}(x,y)\leq \frac{C}{t^{n/2}} \exp\left(-\frac{|x-y|^2}{4t} \right)$ so that $(G)$ holds.
Also we easily have
\begin{equation*}
|\partial_{x_k}p_t(x,y)|\leq \frac{C}{t^{(n+1)/2}} \exp\left(-\frac{|x-y|^2}{ct} \right),\quad t\in (0,1].
\end{equation*}
By using $|x_k|\leq \frac{|x_k-y_k|}{2}+\frac{|x_k+y_k|}{2}$, we obtain the same upper bound
\begin{equation*}
|x_k p_t(x,y)| \leq \frac{C}{t^{(n+1)/2}} \exp\left(-\frac{|x-y|^2}{ct} \right),\quad t\in (0,1].
\end{equation*}
Those upper bounds imply \eqref{Rinfini}, at least for $t\in (0,1]$, as follows 
\begin{eqnarray*}
\| \partial_{x_k} e^{-tL}\|_{\infty\to \infty}+\|x_k e^{-tL}\|_{\infty\to \infty} &\leq&  \sup\limits_{x\in \R^n} \int_{\R^n} |\partial_{x_k}p_{t}(x,y)|+|x_k p_{t}(x,y)| dy \\
& \leq & \frac{C}{\sqrt{t}}.
\end{eqnarray*}
It remains to consider the case $t>1$.  We use \eqref{Rinfini} for $t=\frac{1}{2}$ and the 
exponential decay of $(e^{-tL})_{t\geq 0}$ on $L^\infty(\R^n)$ (note that the spectrum of $L$ is contained in $[1,+\infty)$) to obtain 
\begin{eqnarray*}
\| \partial_{x_k} e^{-tL} \|_{\infty \to \infty}+\| x_k e^{-tL} \|_{\infty \to \infty} &=& \| \partial_{x_k} e^{-\frac{L}{2}} e^{-(t-\frac{1}{2})L}  \|_{\infty \to \infty}+\| x_k e^{-\frac{L}{2}} e^{-(t-\frac{1}{2})L}  \|_{\infty \to \infty} \\
&\le& C \|  e^{-(t-\frac{1}{2})L} \|_{\infty \to \infty} \\
&\le& C' e^{-\frac{1}{2}(t-\frac{1}{2})}\\
& \leq & \frac{C''}{\sqrt{t}}.
\end{eqnarray*}

\medskip
\noindent {3. \it Elliptic operators on domains.} 

Let $\Omega$ be an arbitrary  bounded domain of $\R^n$. We consider the elliptic operator
$$ L = - \sum_{k,l=1}^n  \frac{\partial}{\partial x_k} \left( c_{kl}(x) \frac{\partial}{\partial x_l}  \right) $$
 on $L^2(\Omega)$ and subject to Dirichlet boundary conditions. We assume that  $c_{kl} = c_{lk}  \in L^\infty(\Omega, \R)$ and satisfy the usual  ellipticity  condition
 $$\sum_{k,l=1}^n c_{kl}(x) \xi_k \xi_l \ge \eta | \xi|^2$$
 for some $\eta > 0$ and all $\xi = (\xi_1, \cdots, \xi_n) \in \R^n$ and $x \in \Omega$. It is a well known fact that the heat kernel of $L$ satisfies a Gaussian upper bound (see, e.g. \cite{Davies89} or Chapter 6 in \cite{maati} and the references therein). Therefore, by Theorem \ref{theo1-p}
 $$ \| \nabla \Big{(}  \sum_{k=0}^N  \alpha_k \phi_k \Big{)}  \|_p   \le C \lambda_N \,
\|  \Big{(}  \sum_{k=0}^N  \alpha_k \phi_k \Big{)}  \|_p,
$$
for $p \in [1, 2]$. Here $\phi_k$ are the eigenfunctions  associated  with the eigenvalues $(\lambda_k^2)_k$ of $L$.\\
For $p \in [2, \infty]$,  the later Bernstein inequality  holds under the assumption that $\Omega$ is $C^{1+\varepsilon}$ and the coefficients are $C^\varepsilon$ for some $\varepsilon > 0$.  In this case, the heat kernel of $L$ satisfies the following gradient estimate (and hence the regularity property $(R_p)$ for $1 \le p \le \infty$)
\begin{equation*}
 | \nabla p_t(x,y) |  \le \frac{C}{t^{n/2 + 1/2}} e^{- c\frac{|x-y|^2}{t}}.
 \end{equation*}
 This gradient estimates hold even for complex coefficients and $\nabla p_t(x,y)$ is H\"older continuous. See \cite{EO19}.

\medskip

\noindent {4. \it Compact manifolds.}

Let $M$ be a compact Riemannian manifold  satisfying one of the following hypothesis 
\begin{itemize}
\item[\textit{i)}] The manifold $M$ has no boundary and $\Delta$ will be the Laplace-Beltrami operator.
\item[\textit{ii)}] The manifold $M$ has a smooth boundary which is convex in the sense of \cite[pages 155 and 157]{liyau}. In that case, $\Delta$ will be the Laplace-Beltrami operator with the Neumann boundary condition.
From a geometric point of view, we moreover assume that the Ricci curvature is bounded from below by $-K$ (with $K\geq 0$).
\end{itemize}

In both cases, for the Riemannian measure, the doubling property \eqref{doublin} holds for $n$ being the dimension of $M$.
%In that part, it is not relevant to add a potential to $-\Delta$, so we set $W=0$.
For the null potential $W=0$, it follows from our results  that the Bernstein inequalities $(B_{p,q})$ hold for $1\leq p\leq q \leq +\infty$ as well as the reverse Bernstein inequalities $(RB_q)$ for any $q\in [1,+\infty]$.
As a particular case, we have for any $p\in [1,+\infty]$ and for any eigenvalue $\lambda^2\ge 1$ of $-\Delta$ with eigenfunction $\varphi_\lambda$ the following two-side inequalities
\begin{equation*}
c_1 \lambda \| \varphi_\lambda \|_p \le \| \nabla \varphi_\lambda \|_p \le c_2 \lambda \| \varphi_\lambda \|_p,
\end{equation*}
for some positive constants $c_1$ and $c_2$. This latter inequality is conjectured in \cite{shi2010} for $p=\infty$. Our results answer this conjecture (under the convexity of the boundary in the case of Neumann boundary conditions).\\
 Let us give some details about this. 

\textbf{Case i).} In the boundaryless case, the famous Minakshisundaram theorem implies the following upper bounds
on the the heat kernel $p_t(x,y)$ of $-\Delta$ 
\begin{equation*}
0 \leq p_t(x,y)\leq \frac{C}{t^{n/2}}e^{-c\frac{d(x,y)^2}{t}},\qquad t\in (0,1], \quad(x,y)\in M\times M.
\end{equation*}
We refer to \cite[Chapter VI]{chavel} or \cite[page 204]{BGM71}.
Since $M$ is compact, the volume $V(x,\sqrt{t})$ is equivalent to $\min(1,t^{n/2})$ so that \eqref{doublin} and $(G)$ hold. 

On the other hand, $(R_\infty)$ follows immediately from \eqref{unifo} and the  following gradient estimate  (see \cite{hsu} and references therein) 
\begin{eqnarray*}
|\nabla_x p_t(x,y)| & \leq & C\left(\frac{d(x,y)}{t}+\frac{1}{\sqrt{t}} \right)p_t(x,y)\\
 & \leq & \frac{C}{\sqrt{t}V(x,\sqrt{t})}\left(\frac{d(x,y)}{\sqrt{t}}+1 \right)e^{-c\frac{d(x,y)^2}{t}}\\
 & \leq & \frac{C'}{\sqrt{t}V(x,\sqrt{t})}e^{-c'\frac{d(x,y)^2}{t}}.\\
\end{eqnarray*}
%The above estimates allow us to write
%\begin{eqnarray*}
%\sqrt{t}\left\Vert \nabla e^{t\Delta}\right\Vert_{\infty\rightarrow \infty} & \leq & \sup\limits_{x\in M} \int_{M}\sqrt{t}|\nabla_x p_t(x,y)| d\mu(y) \\
%& \leq & C'\sup\limits_{x\in M} \frac{1}{V(x,\sqrt{t})}\int_{M} e^{-c\frac{d(x,y)^2}{t}}d\mu(y).
%\end{eqnarray*}
%By using \eqref{unifo}, the last upper bound is uniformly bounded with respect to $t>0$. 

\textbf{Case ii).} For the Neumann boundary case, we need the following important result proved in \cite[Theorem 3.2]{liyau} for the heat kernel $h_t(x,y)$ of $-\Delta$ :
\begin{equation*}
0 \le h_t(x,y) \le \frac{C}{\sqrt{V(x,\sqrt{t})V(y,\sqrt{t})}} e^{\varepsilon t} e^{-c \frac{d^2(x,y)}{t}},\qquad t>0,\quad (x,y)\in M\times M.
\end{equation*}
for all $\varepsilon > 0$ and some constants $C$ and $c$.
Note that, we clearly have $V(x,\sqrt{t})\simeq V(y,\sqrt{t})$ (even if $x$ or $y$ belong to the smooth boundary $\partial M$).
Thanks to \cite[Corollary 1.3 and Theorem 1.2]{grigo95} the previous estimates imply the following time derivatives bounds
\begin{equation*}
\left\vert \frac{ \partial h_t(x,y)}{\partial t}\right\vert \le \frac{C}{V(x,\sqrt{t})} \frac{e^{\varepsilon t}}{t} e^{-c \frac{d^2(x,y)}{t}}.
\end{equation*}
On the other hand,  we obtain from \cite{Davies89b} and \cite[Theorem 1.4]{liyau} that 
\begin{eqnarray*}
|\nabla_x h_t(x,y)|^2 & \le & C h_t(x,y)\left\vert \frac{ \partial h_t(x,y)}{\partial t}\right\vert+C\left(K+\frac{1}{t}\right)h_t(x,y)^2\\
|\nabla_x h_t(x,y)|& \leq & \frac{C}{ V(x,\sqrt{t})} \left( K+\frac{1}{t}\right)^{1/2} e^{\varepsilon t} e^{-c \frac{d^2(x,y)}{t}}.
\end{eqnarray*}
The exponential loss $e^{\varepsilon t}$ can actually become an exponential gain by considering $L=-\Delta+2\varepsilon$ instead of $-\Delta$ due to the identity $p_t(x,y)=e^{-2\varepsilon t} h_t(x,y)$ so that we have
\begin{eqnarray}\label{liyau-man}
0\leq p_t(x,y) & \leq & \frac{C}{V(x,\sqrt{t})} e^{-\varepsilon t} e^{-c \frac{d^2(x,y)}{t}}, \\ \nonumber
|\nabla_x p_t(x,y)| &\le & \frac{C}{ V(x,\sqrt{t})} \left( K+\frac{1}{t}\right)^{1/2} e^{-\varepsilon t} e^{-c \frac{d^2(x,y)}{t}}.
\end{eqnarray}
%As for the boundaryless case, such estimates imply the regularity property $(R_\infty)$ for $L$ as follows
%we can bound
%\begin{eqnarray*}
%\sup\limits_{t>0}\sqrt{t}\left\Vert \nabla e^{-tL}\right\Vert_{\infty\rightarrow \infty} & \leq & \sup\limits_{t>0}\sup\limits_{x\in M} \int_{M}\sqrt{t}|\nabla_x p_t(x,y)| %d\mu(y) \\
%& \leq & C\sup\limits_{t>0}\left(\sqrt{t}\left(K+\frac{1}{t}\right)^{1/2} e^{-\varepsilon t} \sup\limits_{x\in M} \frac{1}{V(x,\sqrt{t})}\int_{M} e^{-c\frac{d(x,y)^2}{t}}d\mu(y) \right)\\
%& < & +\infty.
%\end{eqnarray*}
The last bound and \eqref{unifo} imply  $(R_\infty)$ for $L=-\Delta+2\varepsilon$. Note that $(B_p)$ reads as 
%Note also that \eqref{liyau-man} implies the bound $|p_t(x,x)|\leq Ct^{-n/2}$.
%Hence, all the theorems in the introduction are applicable for $L=-\Delta+2\varepsilon$. For instance $(B_p)$ reads for any coefficients $\alpha_0,\dots,\alpha_N$ as follows
\begin{equation*}
\| \nabla \Big{(}  \sum_{k=0}^N  \alpha_k \phi_k \Big{)}  \|_p  \le C (2\varepsilon+\lambda_N) \
\|  \Big{(}  \sum_{k=0}^N  \alpha_k \phi_k \Big{)}  \|_p.
\end{equation*}
%Since for $N=0$ the lower bound is zero, one may assume $N\geq 1$.
%Remembering that $\varepsilon>0$ may be chosen as small as wanted in the Li-Yau inequalities \eqref{liyau-man}, one may fix $2\varepsilon\leq \lambda_1$ %at the beginning. Hence we get $2\varepsilon+\lambda_N\leq 2\lambda_N$ and $(B_p)$ finally holds for $-\Delta$.
We can remove $\varepsilon$ from this estimate by taking $2\varepsilon\leq \lambda_1$. We deal similarly with the reverse Bernstein inequality. 

\bibliographystyle{alpha}
%\bibliography{bernstein-bib}

 \newcommand{\etalchar}[1]{$^{#1}$}

\end{document}